\documentclass[11pt]{amsart}

\usepackage{mathtools}
\usepackage{amsmath}
\usepackage{amssymb, stmaryrd, bm}
\usepackage{amsthm}
\usepackage{graphicx}
\usepackage{tikz}
\usepackage{xcolor}
\usepackage[shortlabels]{enumitem}
\usepackage{algorithm}
\usepackage[colorinlistoftodos]{todonotes}
\usepackage{tcolorbox}
\newtcolorbox{simplebox}{
  boxrule=0pt,           
  frame hidden,          
  colback=white,         
  opacityback=0,         
  fontupper=\itshape,    
  left=12pt, right=12pt, 
  top=6pt, bottom=6pt,
  before skip=\medskipamount,  
  after skip=\medskipamount     
}

\usepackage[foot]{amsaddr}
\usepackage[pdftex,colorlinks,backref=page,citecolor=blue]{hyperref}

\allowdisplaybreaks

\usepackage[margin=1.2in]{geometry}
\usepackage{bookmark}

\usepackage[noabbrev,capitalize]{cleveref}
\crefname{equation}{}{}
\crefformat{enumi}{#2#1#3}
\crefrangeformat{enumi}{#3#1#4 to~#5#2#6}
\crefmultiformat{enumi}{#2#1#3}
{ and~#2#1#3}{, #2#1#3}{ and~#2#1#3}
\crefname{claim}{Claim}{Claims}
\Crefname{claim}{Claim}{Claims} 

\def\E{\mathbb{E}}
\def\var{\textrm{Var}}

\def\N{\mathbb{N}}

\def\R{\mathbb{R}}
\def\P{\mathbb{P}}

\def\eps{\varepsilon}

\def\cA{\mathcal {A}}

\def\cE{\mathcal {E}}

\def\cH{\mathcal {H}}

\def\cS{\mathcal {S}}
\def\cT{\mathcal {T}}
\def\cI{\mathcal {I}}

\def\bI{\mathbb {I}}

\def\cW{\mathcal {W}}

\def\sf{$\sigma$-field}
\def\1{\mathbf{1}}

\def\lam {\lambda}

\def\tce{t_c + \eps}
\def\tce2{t_c + \frac{\eps}{2}}

\newcommand{\bb}[1]{\mathbb{#1}}
\newcommand{\fm}[1]{\mathfrak{#1}}
\renewcommand{\sf}[1]{\mathsf{#1}}

\newcommand{\ol}[1]{\overline{#1}}

\newcommand{\of}[1]{\of{\fm{#1}}}

\newcommand{\mc}[1]{\mathcal{#1}}
\newcommand{\mb}[1]{\mathbb{#1}}
\newcommand{\mr}[1]{\mathrm{#1}}

\renewcommand{\Pr}{\operatorname{\mathbb{P}}}

\newcommand{\codeg}{{\mr{codeg}}}

\newtheorem*{theorem*}{Theorem}
\newtheorem{theorem}{Theorem}[section]
\newtheorem{lemma}[theorem]{Lemma}
\newtheorem{conjecture}[theorem]{Conjecture}
\newtheorem{observation}[theorem]{Observation}
\newtheorem{corollary}[theorem]{Corollary}
\newtheorem{claim}[theorem]{Claim}

\newtheorem*{fact*}{Fact}

\newtheorem{remark}[theorem]{Remark}

\theoremstyle{definition}

\newtheorem{definition}[theorem]{Definition}
\newtheorem{notation}[theorem]{Notation}

\numberwithin{equation}{section}

\title[The replica symmetric solution for hypergraph independent sets]{The replica symmetric solution for hypergraph independent sets in the critical regime}

\author{Matthew Jenssen}
\author{Will Perkins}
\author{Aditya Potukuchi}
\author{Michael Simkin}

\address{King's College London, Department of Mathematics}
\email{matthew.jenssen@kcl.ac.uk}

\address{Georgia Institute of Technology, School of Computer Science}
\email{math@willperkins.org}

\address{Max Planck Institute for Software Systems}
\email{apotu@mpi-sws.org}

\address{Hebrew University of Jerusalem, Einstein Institute of Mathematics and Rachel and Selim Benin School of Computer Science and Engineering}
\email{michael.simkin@mail.huji.ac.il}

\begin{document}

\begin{abstract}
    We prove a variational formula for the logarithmic asymptotics of a non-existence probability in a broad class of combinatorial problems such as avoiding cliques in random graphs and $k$-term arithmetic progressions in random subsets of integers. These results follow from a formula for the probability that a binomial random subset of the vertices of a locally sparse hypergraph is an independent set. The formula holds throughout the critical regime, interpolating between the regimes in which Janson's inequality and the method of hypergraph containers give the respective asymptotics. The formula is the replica-symmetric Bethe free energy formula from statistical physics applied to a hypergraph hardcore model.

    The proof uses two new techniques.  The first, for the upper bound, involves revealing a small `window' of a random independent set and then applying entropy methods.  The second, for the lower bound, involves the analysis of a random greedy algorithm guided by a Belief Propagation fixed point.  

    In the case of  avoiding cliques in random graphs, the variational formula can be expressed as an optimization problem over graphons. Moreover, typical random graphs conditioned on not containing $K_r$, when suitably normalized, approach the set of optimizing graphons in cut distance.
\end{abstract}

\maketitle

\section{Introduction}
\label{secIntro}

What is the probability that the random graph $G(n,p)$ has no triangles? Estimating non-existence probabilities like this (random graphs avoiding another subgraph $H$, a random set of integers avoiding a $k$-term arithmetic progression, a random subset of an Abelian group being sum-free, among many other natural examples) is a long-studied problem in combinatorics and probability.  Problems of this type can be encoded as problems about independent sets in hypergraphs: what is the probability that a $p$-random vertex subset of a given hypergraph induces no edges? Powerful and versatile tools in combinatorics have been developed precisely to attack these problems.

On one hand, when $p$ is small (as a function of parameters of the hypergraph), Janson's inequality~\cite{janson1987uczak,janson1990poisson} and the Harris inequality give the logarithmic asymptotics of the non-existence probability, matching a Poisson distribution.  On the other hand, when $p$ is large, the method of hypergraph containers applies~\cite{saxton2015hypergraph,balogh2015independent} and shows the non-existence probability is closely connected to an extremal problem.  For example, in the case of the random graph $G(n,p)$ avoiding a triangle (with $X$ denoting the number of triangles), we have 
\begin{itemize}
    \item (Subcritical) When $p = o(n^{-1/2})$, $\log \P (X=0) = - (1+o(1)) \E X$~\cite{janson1987uczak} (here and in what follows, all logarithms are natural);
    \item (Supercritical) When $p = \omega(n^{-1/2})$,  $\log \P (X=0) = (1+o(1))\log \P (G \text{ bipartite}) = (1+o(1)) \frac{n^2}{4}\log(1-p)$~\cite{luczak2000triangle}.
\end{itemize}
In particular, we see that the subcritical formula is  generic (the `Poisson paradigm') while the supercritical formula is related to an optimization problem (Mantel's theorem in the case of avoiding triangles).  
What remains mysterious is the \textit{critical regime}, between the range of applicability of Janson's inequality and hypergraph containers (for avoiding triangles, $p= \Theta(n^{-1/2})$).  This dichotomy between subcritical and supercritical results and methods is common to many  combinatorial problems of avoiding a substructure in a random structure. 

Previous work of the authors established logarithmic asymptotics in some restricted special cases: triangles in $G(n,p)$ in a small portion of the critical regime~\cite{jenssen2024lower}; the general hypergraph setting again in a small portion of the critical regime~\cite{jenssen2026non}.

Here we give a complete solution to the problem in great generality, expressing the logarithmic asymptotics of the non-existence probability as a variational problem.  This variational problem is maximizing the \textit{Bethe free energy} over all assignments of `marginals' in $[0,1]$ to vertices of the hypergraph encoding the specific problem.  Below we will give background on the Bethe free energy,  which arose in statistical physics and the study of inference on graphical models, and explain why it is relevant in these combinatorial and probabilistic problems.

In the more structured setting of avoiding subgraphs in random graphs, by taking limits  we obtain a variational formula over graphons, with the additional consequence of proving the existence of the  limit of the scaled log avoidance probabilities. Furthermore, we prove that the set of optimizing graphons determines the typical structure of random graphs conditioned to avoid the subgraph in question.

The precise setup we consider is the following. Let $\mc H = (V,E)$ be a $k$-uniform hypergraph of maximum vertex degree $\Delta$, and let $\mc I = \mc I(\mc H)$ be the set of independent sets of $\mc H$; that is, vertex subsets that induce no edge. For $p \in (0,1)$, let $S \subseteq V$ be chosen by including each vertex $v$ independently with probability $p$.    Our main goal is to compute first-order asymptotics of the logarithm of the non-existence probability,  $\log \P_p[S \in \mc I]$, as the number of vertices of $\mc H$ tends to infinity.  

To state the main results, we need to define the Bethe free energy  as well as a class of locally tree-like hypergraphs to which our results apply.  The basic objects we work with are assignments of `marginals' to the vertices of the hypergraph $\mc H$, representing the probability each vertex is in $S$ given that $S \in \mc I$.    The Bethe free energy is a functional mapping a set of marginals to a real number.  Adapted to our setting, it is the following.
\begin{definition}
\label{defBethe}
    Let $\mc H = (V,E)$ be a hypergraph and let $p \in (0,1)$. The \textit{Bethe free energy} on $\mc H$ with parameter $p$, $\mc B_{\mc H,p} : [0,1]^V \to \R$ is defined by
    \[
    \mc B_{\mc H,p}(x) = - \sum_{e \in E} \prod_{v \in e} x_v + \sum_{v \in V} \left( x_v - p - x_v \log\left(\frac{x_v}{p}\right)\right).
    \]
\end{definition}

Our main theorem is that the log-asymptotics of the probability that a $p$-random binomial subset of $V$ is independent is governed by the Bethe free energy, as long as locally $\mc H$ resembles a hypertree. In many cases this is already known for $p$ either large or small, in which cases the Bethe free energy reduces to the mean-field approximation or the Poisson approximation respectively. Our contribution is bridging the gap by studying the \textit{critical regime},  in which these two approximations  balance in a delicate way.  We next define a notion of `resembling a hypertree' and define the critical regime.

\begin{notation}
    Let $\mc H=(V,E)$ be a hypergraph. For a vertex $v \in V$ we define $E_v \coloneqq \{ e \subseteq V \setminus \{v\} : e \cup \{v\} \in E\}$, the collection of sets completing $v$ to an edge. We write $\deg_{\mc H}(v) = |E_v|$ for the degree of $v$. We also write $\Delta(\mc H) = \max_v \deg_{\mc H}(v)$ for the maximum hypergraph degree. For $1 \leq \ell \leq k$ we write $\Delta_\ell(\mc H)$ for the maximum $\ell$-degree, i.e., the maximum over all sets $S \in \binom{V}{\ell}$ of the number of edges containing $S$. For a pair of distinct vertices $u,v \in V$ we write $\codeg_{\mc H}(u,v) = |E_u \cap E_v|$ for the number of mutual completions of $u$ and $v$ to edges.   Usually, the hypergraph $\mc H$ will be clear from context. In this case we may write $\Delta_\ell$ in place of $\Delta_\ell(\mc H)$, and $\Delta$ in place of $\Delta(\mc H)$.
\end{notation}

\begin{definition}\label{def:LTL}
    Let $\eta > 0$. An $n$-vertex $k$-uniform hypergraph $\mc H=(V,E)$ is \textit{$\eta$-locally tree-like} if:
    \begin{enumerate}[{\bfseries{A\arabic{enumi}}}]
        \item\label{itm:setup codegree bound} $\max_{u\neq v} \codeg_{\mc H}(u,v) \leq \eta \Delta(\mc H)$,
        \item\label{itm:setup Delta ell UB} for each $\ell \in \{2,\ldots,k-1\}$, $\Delta_\ell(\mc H) \leq \eta\Delta(\cH)^{(k-\ell)/(k-1)}$.
    \end{enumerate}
    For $c\in (0,1)$, a probability density $p \in (0,1)$ is \textit{$c$-critical with respect to $\mc H$} if
    \begin{enumerate}[{\bfseries{A\arabic{enumi}}},resume]
        \item\label{itm:setup p is critical} $p \Delta^{1/(k-1)} \in [c,1/c]$.
    \end{enumerate}
\end{definition}
We remark that if $\mc H$ has average degree $\Theta(\Delta)$, the critical regime is precisely when the expected number of vertices in a $p$-random binomial subset of $V$ (i.e., $np$) is of the same order as the expected number of induced edges ($\Theta(n\Delta p^k)$).
\begin{theorem}\label{thm:Bethe exactness}
    Fix $k \geq 2$ and $c\in (0,1)$. Let $\cH=(V,E)$ be an $n$-vertex, $k$-uniform, $\eta$-locally tree-like hypergraph where $\eta\to 0$ and $\Delta(\cH) \to \infty$ as $n\to\infty$.
    Suppose $p$ is $c$-critical with respect to $\mc H$. Let $\mc I$ be the collection of independent sets in $\mc H$. Let $S$ be a $p$-random binomial subset of $V$. Then the probability that $S$ is an independent set satisfies
    \[
    \log  \Pr_p [S \in \mc I]  = \max_{x\in [0,1]^V} \mc B_{\mc H,p}(x) + o(np).
    \]
\end{theorem}

In Section~\ref{subsecBP} we show that the maximization can be restricted to those $x \in [0,1]^V$ that are \textit{Belief Propagation fixed points} (defined below in Section~\ref{subsecBP}).
\cref{thm:Bethe exactness} establishes the \textit{replica symmetric formula} for the logarithmic asymptotics throughout the critical regime, in the sense from statistical physics that the Bethe free energy is accurate (e.g.~\cite{mezard2009information,dembo2014replica}).  In \cref{subsecRelated} below we contrast this to other notions of replica symmetry used in the large deviations literature.

Note that the Bethe free energy consists of two terms; one is a Janson-like term which resembles an expected number of edges, the other is a relative entropy term which resembles that in the variational problem for the supercritical regime~\cite{kozma2023lower}. \cref{thm:Bethe exactness} shows that these two terms are in balance in the critical regime.

\subsection{Applications}

\subsubsection{Avoiding subgraphs in random graphs}\label{sec:intro avoiding subgraphs}

A central application of \cref{thm:Bethe exactness} --- indeed, the original motivation for this paper --- is to determine the rate function (that is, the normalized limit of the log probability) for the non-existence of a copy of $K_r$ in $G(n,p)$ in the critical regime. Here, the critical regime is $p = \Theta(n^{-2/(r+1)})$. To state our result we use standard graphon terminology (see, for example, \cite{lovasz2012large}). Let $\mc W$ be the collection of symmetric, Borel measurable functions $w:[0,1]^2\to[0,1]$. We write $\delta_\Box(\cdot,\cdot)$ for the \textit{cut distance} between two graphons. For a graph $G$, we write $\tilde G$ for the step graphon it induces. For $w \in \mc W$ and a graph $H$ with vertex set $\{1,2,\ldots,v(H)\}$ we define
\[
t(H,w) \coloneqq \int_{[0,1]^{v(H)}} \prod_{ij \in E(H)} w(x_i,x_j) dx_1\cdots dx_{v(H)}.
\]
This is the homomorphism density of $H$ in $w$. We also define the entropy functional
\[
I(w) \coloneqq \frac{1}{2} \iint_{x,y \in [0,1]} \left( w(x,y)\log(w(x,y)) + 1 - w(x,y) \right) dxdy
\]
(as usual, $0\log0\coloneqq0$). Then, for $r \geq 3$ and $c>0$, we define the $K_r$-avoidance Bethe functional
\[
\mc B^{K_r}_c(w) \coloneqq - \frac{1}{r!} c^{\binom{r}{2}} t(K_r,w) - c I(w).
\]
The two terms in $\mc B^{K_r}_c(w)$ describe the competing costs of changing the edge-density profile of $G(n,p)$ and suppressing copies of $K_r$: $I(w)$ is the entropy cost of deviating from the uniform graphon $w \equiv 1$, and $t(K_r,w)$ penalizes the expected number of copies of $K_r$. Finally, we define
\[
F_{K_r}(c) \coloneqq \max_{w\in\mc W} \mc B^{K_r}_c(w)
\]
(we prove that the maximum exists in \cref{sec:Graphon}).

\begin{theorem}\label{thm:graphon clique avoidance}
    Fix an integer $r \geq 3$, a constant $c>0$ and let $p=cn^{-2/(r+1)}$. Then
    \[
   f_r(c) :=  \lim_{n \to \infty} n^{-2r/(r+1)} \log (\Pr[K_r \nsubseteq G(n,p)]) = F_{K_r}(c).
    \]
    Moreover, if $G \sim G(n,p)$ conditioned on not containing $K_r$ as a subgraph, then whp there exists a graphon $w$ that maximizes $\mc B_c^{K_r}$ and such that $\delta_\Box(\tilde G/p,w)=o(1)$.
\end{theorem}

\cref{thm:graphon clique avoidance} establishes the existence of the limit defining $f_r(c)$ throughout the critical regime, answering in the affirmative Question 1  from~\cite{jenssen2024lower} (in the case $\eta=0$).  In~\cite{jenssen2026non} it is shown that $f_r(c)$ has at least one point of non-analyticity on the interval $(0, \infty)$.

In \cref{sec:Graphon} we formulate and prove a more general version of \cref{thm:graphon clique avoidance} that covers the probability that $G(n,p)$ avoids a fixed graph $H$ for any strictly $2$-balanced $H$. We also prove some results and state a conjecture regarding the typical structure of $K_r$-free graphs in the critical regime and the optimization problem defining $F_{K_r}(c)$.

\subsubsection{Avoiding $k$-APs}
Fix $k\geq 3$.
A second application of \cref{thm:Bethe exactness} is to the study of random integer sets avoiding $k$-term arithmetic progressions ($k$-APs). Let $[n] = \{1,\ldots,n\}$ and let $[n]_p$ denote the binomial random model on $[n]$ with density $p$. The natural constraint hypergraph here is $\mc H_k(n)$, which has vertex set $[n]$ and whose hyperedges are the $k$-APs in $[n]$.

The critical regime is when $p = \Theta(n^{-1/(k-1)})$. The story is similar to that of clique-free random graphs: Letting $X(S)$ be the number of $k$-APs in a set $S$, Janson's and Harris's inequalities (in the subcritical regime) and hypergraph container methods \cite{saxton2015hypergraph,balogh2015independent,kozma2023lower} (in the supercritical regime) imply that
\[
\log(\Pr_{S \sim [n]_p}[S \text{ is $k$-AP-free}]) = \begin{cases}
    - (1+o(1))\E[X(S)] & p = o(n^{-1/(k-1)})\\
    (1+o(1))n\log(1-p) & p = \omega(n^{-1/(k-1)}).
\end{cases}
\]

In \cite{jenssen2026non}, the authors determined the log asymptotics of $\Pr[S \text{ is $k$-AP-free}]$ whenever $p < c_k n^{-1/(k-1)}$, for a small constant $c_k$. It is straightforward to show that the natural constraint hypergraph for $k$-term arithmetic progressions in $[n]$ is $n^{-\delta}$-locally tree-like, for some $\delta = \delta(k)$ (cf.\ \cite[Lemma 8.1]{jenssen2026non}). Thus, \cref{thm:Bethe exactness} determines the asymptotic log probability that $[n]_p$ contains no $k$-term arithmetic progression throughout the critical regime, in terms of an optimization problem over assignments of marginals to the integers $[n]$.

Similar to the subgraph avoidance case, it is natural to seek a limiting continuous formulation. Indeed, weights on $[n]$ naturally correspond to step functions from $[0,1]$ to $[0,1]$, and one can formulate an appropriate variational problem on the limit space of such functions. Unfortunately, there are delicate convergence issues that do not appear in the graphon case: The most natural topology here is the weak--$*$ topology, with respect to which the limit space is compact. However, the functional that `counts' arithmetic progressions is not weak--$*$ continuous. Thus, we leave as open problems to prove that, for every $c>0$, the limit
\[
\lim_{n\to\infty}n^{1/(k-1)-1} \log(\Pr_{S \sim [n]_{cn^{-1/(k-1)}}}[ S \text{ is $k$-AP-free}])
\]
exists and to determine its value.

\subsection{Related work}
\label{subsecRelated}

The study of $\P_p[S \in \mathcal I]$, and particularly the special case of $G(n,p)$ avoiding a copy of a fixed graph $H$ has a long history, with connections to several different areas of mathematics.  One of the earliest and most influential results on this problem was the theorem of Erd\H{o}s, Kleitman, and Rothschild~\cite{erdosasymptotic} showing that almost every triangle-free graph is bipartite, or in other words, in $G(n,1/2)$ the first-order asymptotics of the probability of triangle-freeness coincide with the probability of being bipartite.   This result has since been extended and generalized in many different directions, including moving from $p=1/2$ to other values of $p$, treating other subgraphs besides cliques, studying other types of large deviations (the upper and lower tail problems).  Focusing just on lower-tail results since they are directly relevant to the current paper, there is a line of work identifying the approximate number and typical structure of $H$-free graphs~\cite{erdosasymptotic,promel1996asymptotic,janson1987uczak,promel1996counting,luczak2000triangle,osthus2003densities,balogh2015independent,saxton2015hypergraph,balogh2016typical,jenssen2023evolution} that are the immediate predecessors of ours. For the more general lower-tail problem, the subcritical regime again follows Poisson behavior~\cite{janson1990poisson,janson2016lower} while understanding the supercritical behavior involves solving an optimization problem.  Results showing that the behavior reduces to such an optimization problem include~\cite{chatterjee2011large,chatterjee2016nonlinear,augeri2020nonlinear,eldan2018gaussian}, culminating in the work of Kozma and Samotij~\cite{kozma2023lower} establishing this reduction through the entire supercritical regime. Building on the work of Kozma and Samotij, Chin~\cite{chin2025structure} obtained results on the typical structure of sparse random graphs conditioned on lower-tail subgraph counts. Solving the relevant optimization problem for  lower-tail problems turns out to be very challenging, but some results are known~\cite{zhao2017lower}.
In the critical regime, the only result known until recently was the order of magnitude of the log probability~\cite{janson1987uczak}; the two recent  papers~\cite{jenssen2024lower,jenssen2026non} find the logarithmic asymptotics for sufficiently small critical densities (in fact precisely in a regime where the Bethe free energy has a unique stationary point, at the global maximum).

As we describe in more detail below, Belief Propagation and the Bethe free energy are heuristics for approximating vertex marginals and the log partition function of graphical models. These heuristics are exact on acyclic graphs and hypergraphs.  A major direction of study is understanding when these heuristics are approximately accurate in hypergraphs that may contain cycles.  For statistical physics models on sparse random graphs (say random $\Delta$-regular graphs for $\Delta$ constant), this has been verified in a number of cases.  Roughly, results confirming the replica symmetric prediction (that is, the log partition function is well approximated by the maximum of the Bethe free energy) fall into two categories: for a specific model when interactions are weak enough or $\Delta$ small enough~\cite{bapst2016harnessing,coja2018belief,coja2018information,coja2018charting,biswas2025replica}; for classes of models, like ferromagnetic spin models, for all values of parameters~\cite{DemboMontanari2010,dembo2013factor,dembo2014replica,galanis2016inapproximability,cai2016hardness,sly2014counting}.  In particular, because the replica symmetric solution is an optimization problem, a phase transition can still occur even if a model is replica symmetric for all parameters (as occurs in, say, the ferromagnetic Ising and Potts models on random graphs).  In fact, we conjecture this type of phase transition occurs in the problem of avoiding cliques in random graphs.

Theorem~\ref{thm:Bethe exactness} shows that the `replica symmetric cavity method' (e.g.,~\cite{mezard2009information,dembo2014replica}) is accurate for this hypergraph independent set problem throughout the entire critical regime.  This may well be  surprising to a reader familiar with either combinatorial non-existence probabilities and large deviations in random graphs or with the cavity method for spin models on random graphs and hypergraphs, so let us make a couple remarks to clarify (though even after clarification we find this remarkable!).

First, in some of the literature on large deviations in random graphs~\cite{chatterjee2011large,lubetzky2015replica,zhao2017lower,mukherjee2020replica}, a variational problem is often called replica symmetric when its optimizer is a constant graphon (equivalently, when the large deviation event is realized by a uniform Erdős–Rényi tilt). This  should not be confused with replica symmetry in the statistical physics spin-glass and cavity method sense, which concerns symmetry among replicas or the distribution of overlaps of independent samples from a Gibbs measure. A constant graphon optimizer  is instead a statement of homogeneity of marginals; a non-constant  optimizer does not by itself constitute replica-symmetry breaking in the statistical physics sense.  In this paper we adopt the original statistical physics sense.

Second, to readers familiar with the cavity method applied to sparse random graphs, it is well known that the hardcore model on a random $\Delta$-regular graph does undergo a replica symmetry breaking phase transition (and is also typically locally treelike).  This apparent inconsistency is explained by the fact that when we work in the $\Delta \to \infty$ regime and compute the asymptotics of $\log \P_p [ S \in \mc I]$ rather than $\log Z_{\mc H}(\lambda)$, the additional term $np$ in the formula dominates $\log Z_{\mc H}(\lambda)$ when there is replica symmetry breaking.

\subsection{Concurrent independent work}

Partway through working on this project the authors discovered that Matthew Kwan and Huy Tuan Pham were independently working on very similar problems. We wish to thank them for friendly discussions and for coordinating posting of our papers.

In particular, Kwan and Pham \cite{kwan2026nonlinear} essentially prove \cref{thm:Bethe exactness} (with a mild quantitative assumption on $\eta$ required for the lower bound) and the non-structural statements of \cref{thm:graphon clique avoidance,thm:H non-existence} (they discuss the prospect of proving structural results, and this seems well within reach of their methods). Their paper goes beyond non-existence and obtains a variational formula for lower tail probabilities throughout the critical regime. They additionally prove analogous results for hypergeometric models. Their paper also solves the variational problems associated to avoiding a subgraph $H$ in $G(n,p)$ when $H$ is either strictly $2$-balanced bipartite with the Sidorenko property or a triangle. This proves, respectively, the even-cycle case of \cite[Conjecture 1]{jenssen2026non} and the $r=3$ case of \cref{con:block structure} below. Together with the structural statement in \cref{thm:H non-existence} this gives a very satisfying description of the behavior of $G(n,p)$ conditioned on $H$-freeness throughout the critical regime.

While the papers' methods share some similarities, we take a statistical physics perspective while theirs is more combinatorial. In particular, as we describe below, both our lower and upper bounds use the Gibbs variational formula for the hardcore model as a starting point. We use a Belief Propagation guided random process to prove the lower bound and a local entropy analysis for the hardcore model for the upper bound. Interestingly, the corresponding quantitative calculations in our proofs below for the lower and upper bounds in Theorem~\ref{thm:Bethe exactness} mirror each other quite closely.

\subsection{Organization}

The rest of this paper is organized as follows. In \cref{sec:preliminaries} we collect some basic facts about the hardcore model and the Bethe free energy. In~\cref{sec:motivation} we give an overview of the proof of Theorem~\ref{thm:Bethe exactness}.
In \cref{sec:upper bound} we prove the upper bound in \cref{thm:Bethe exactness} and in \cref{sec:lower bound} we prove the corresponding lower bound. In \cref{sec:hypergraph-structure} we prove a structural result for the hypergraph hardcore model which is a key step towards the structural statement contained in \cref{thm:graphon clique avoidance}.
In \cref{sec:Graphon} we study the subgraph-avoidance case of \cref{thm:Bethe exactness}: we prove and generalize \cref{thm:graphon clique avoidance}, and consider the structure of typical graphs that avoid a fixed subgraph.

\section{Preliminaries}\label{sec:preliminaries}

\subsection{The hardcore model, its partition function, and non-existence probabilities}
\label{subsecHardcore}

For $p\in (0,1)$, let $\mu_{p}$ denote the $p$-random binomial distribution on subsets $S$ of $V$; that is, each vertex is included in $S$ independently with probability $p$. We write $\mu_{\mc H,p}$ for $\mu_p$ conditioned on the event that $S \in \mc I(\mc H)$.

The distribution $\mu_{\mc H,p}$ coincides exactly with the \textit{hardcore model} on the hypergraph $\mc H$ at fugacity $\lambda  = \lambda(p) \coloneqq p/(1-p)$. That is, $\mu_{\mc H,p}$ samples a set $S \in \mc I(\mc H)$ with probability proportional to $\lambda^{|S|}$. We write $Z_{\mc H}(\lambda)$ for partition function of the hardcore model on $\mc H$ at fugacity $\lambda$, i.e., $Z_{\mc H}(\lambda) \coloneqq \sum_{S \in \mc I(\mc H)} \lambda^{|S|}$.

We now record some simple connections between the probability that a sample from $\mu_p$ is independent, the entropy of a sample from $\mu_{\mc H,p}$, and the partition function $Z_{\mc H}(\lambda)$. Throughout this section we fix a hypergraph $\cH$ on $n$ vertices, $p\in (0,1)$ and let $\lam= p/(1-p)$.

\begin{observation}\label{obs:hardcore partition and log probability}
If $S \sim \mu_p$ then  
\[
\log (\Pr [S \in \mc I(\mc H)]) = \log Z_{\mc H}(\lambda) - n \log(1+\lambda)
\]
\end{observation}
\begin{proof}
This follows straightforwardly from unpacking the definitions:
\[
\Pr [S \in \mc I(\mc H)]
=
\sum_{I\in \cI(\cH)}p^{|I|}(1-p)^{n-|I|}= (1-p)^nZ_{\cH}(p/(1-p))\, .\qedhere
\]
\end{proof}
Key to our proofs is the following Gibbs variational formula.
\begin{observation}\label{obs:Gibbs}
\[\
    \log Z_{\cH}(\lambda)
  =
  \max_{\nu}
  \left\{
  H(\nu)+\log\lambda\,
  \mathbb E_{\nu}|\bb I|
  \right\},
  \]
    where the maximum is over all probability distributions $\nu$ supported on $\cI(\cH)$ and the maximum is attained at $\nu=\mu_{\cH,p}$.
\end{observation}
\begin{proof}
Let $\mu=\mu_{\mc H,p}$ and recall that $\mu(I)=\lambda^{|I|}/Z_{\cH}(\lam)$ for $I\in \cI(\cH)$. Note that by non-negativity of the Kullback-Leibler divergence
    \begin{align*}
        0\leq D(\nu||\mu)
        =
        \sum_{I\in \cI(\cH)}
        \nu(I)\log \frac{\nu(I)}{\mu(I)}
        =
        -H(\nu) - \log \lam\, \E_\nu|\mb I| + \log Z_{\cH}(\lambda)\, .
    \end{align*}
    Rearranging gives $\log Z_{\cH}(\lambda)\geq H(\nu) + \log \lam\, \E_\nu|\mb I|$ and we have equality precisely when $\nu=\mu$.    
\end{proof}

We will also use the following basic `spatial Markov' property of the hardcore model. 
\begin{observation}\label{obs:spatial-markov}
For every set $R\subseteq V$ and function
$F\colon\mathcal I(\mc H)\to\mathbb R$, we have
\[
\mathbb E\left[\mathbf 1_{\{R\subseteq \mb I\}}F(\mb I)\right]
=
p^{|R|}
\mathbb E\left[
\mathbf 1_{\{ \mb I\cup R\in\mathcal I(\mc H)\}}F(\mb I\cup R)
\right].
\]
In particular,
\[
\P(R\subseteq \bb I)\leq p^{|R|}\, .
\]
\end{observation}

\begin{proof}
    Recall that $\lambda/p = 1+\lambda$. By definition,
    \begin{align*}
     \E \left[\mathbf 1_{\{R\subseteq \mb I\}}F(\mb I)\right] 
     & = \frac{1}{Z_{\mc H}(\lambda)} \sum_{S \subseteq V \setminus R : S \cup R \in \mc I(\mc H)} \lambda^{|S|+|R|} F(S \cup R)\\
     & = \frac{1}{Z_{\mc H}(\lambda)} p^{|R|} \sum_{S \subseteq V \setminus R : S \cup R \in \mc I(\mc H)} \lambda^{|S|} F(S \cup R) \left( \frac{\lambda}{p} \right)^{|R|}\\
     & = \frac{1}{Z_{\mc H}(\lambda)} p^{|R|} \sum_{S \subseteq V \setminus R : S \cup R \in \mc I(\mc H)} \lambda^{|S|} F(S \cup R) (1+\lambda)^{|R|}\\
     & = \frac{1}{Z_{\mc H}(\lambda)} p^{|R|} \sum_{S \subseteq V \setminus R : S \cup R \in \mc I(\mc H)} \lambda^{|S|} F(S \cup R) \sum_{W \subseteq R} \lambda^{|W|}\\
     & = \frac{1}{Z_{\mc H}(\lambda)} p^{|R|} \sum_{S \subseteq V: S \cup R \in \mc I(\mc H)} \lambda^{|S|} F(S \cup R)\\
     & = p^{|R|} \E \left[\mathbf 1_{\{ \mb I\cup R\in\mathcal I(\mc H)\}}F(\mb I\cup R)\right],
    \end{align*}
    as desired.
\end{proof}

\subsection{Bethe free energy and belief propagation fixed points}\label{subsecBP}

Here we give some background on Belief Propagation and the Bethe free energy.  For background on these topics see, e.g.,~\cite{mezard2009information,yedidia2003understanding}.  At a high level, Belief Propagation is a heuristic message-passing algorithm for computing marginals (the probabilities that variables or vertices receive particular spins) in graphical models.  The heuristic is based on the assumption that conditioned on the status of a vertex (say, it is unoccupied in the hardcore model), the spins of  its neighbors are independent.  This assumption is easily seen to be true on a tree (or hypertree), and so Belief Propagation is exact in this case.  The Bethe free energy is a formula for $\log Z$ that uses this same assumption to compute $H(\nu)$ in the Gibbs variational formula (\cref{obs:Gibbs}) and again on a hypertree this heuristic  formula is exact.  The hypergraphs we deal with here are not acyclic, but the assumption of $\eta$-locally tree-like means they look similar to hypertrees in local neighborhoods.  Moreover, a key part of our proofs is verifying an approximate version of the conditional independence assumption that underlies the two heuristics (see~\cref{sec:motivation} for more).  Note that this conditional independence assumption is more sophisticated than the `mean-field' heuristic which assumes all spins are independent without conditioning; compared to results in the supercritical regime such as~\cite{kozma2023lower} in which the mean-field approximation is proved accurate, our results show that the more sophisticated BP heuristic is necessary inside the critical regime.

Before we proceed, let us briefly examine the maximizer(s) of $\mc{B}_{\mathcal{H},p}$. Throughout this section and Sections~\ref{sec:motivation},~\ref{sec:upper bound}
and~\ref{sec:lower bound}, we work in the setup of
Theorem~\ref{thm:Bethe exactness}. Thus $k\geq2$ and
$c\in(0,1)$ are fixed, $\eta=\eta(n)\to0$,
$\Delta=\Delta(\cH)\to\infty$, and $p$ is $c$-critical.
We write $\lambda=p/(1-p)$. Implicit constants may depend
on $k,c$ and on any additional fixed parameters explicitly
indicated.

We first define a Belief Propagation fixed point.
\begin{definition}\label{def:BP fixed point}
    Let $\mc H = (V,E)$ be a hypergraph and let $p \in (0,1)$. We write $E_v \coloneqq \{ e \setminus \{v\} : v \in e \in E \}$ for the edges containing $v$, with $v$ removed. A set of marginals $x \in  [0,1]^V$ is a Belief Propagation (BP) fixed point with respect to $\mc H$ and $p$ if, for every $v \in V$, 
    \begin{equation}\label{eq:BP fixed point definition}
    x_v = p \exp \left( - \sum_{e \in E_v} \prod_{u \in e} x_u \right).
    \end{equation}
\end{definition}

 The formulas in \cref{def:BP fixed point} and~\cref{defBethe} are somewhat simplified from the definitions for bounded-degree factor graphs since we work in a specific asymptotic regime; the Bethe free energy also includes an additional $-np$ term to account for the difference between $\log (\Pr [S \in \mc I(\mc H)])$  and $\log Z_{\mc H}(\lambda)$. 

\begin{observation}\label{obs:Bethe free energy properties}
    Let $x \in [0,1]^V$. Then $x$ is a BP fixed point with respect to $\mc H$ and $p$ iff it is a stationary point of $\mc B_{\mc H,p}$. Moreover, $\mc B_{\mc H,p}$ attains its maximum on a BP fixed point. Finally, if $x$ is a BP fixed point, then $x_v \leq p$ and $x_v = \Omega(p)$ uniformly over $v$.
\end{observation}

\begin{proof}
    Taking partial derivatives:
    \[
    \frac{\partial \mc B_{\mc H,p}}{\partial x_v} = -\sum_{e \in E_v} \prod_{u \in e} x_u + 1 - \log \left( \frac{x_v}{p} \right) - 1 = -\sum_{e \in E_v} \prod_{u \in e} x_u - \log \left( \frac{x_v}{p} \right)
    \]
    By definition, $x$ is a stationary point of $\mc B_{\mc H,p}$ iff for every $v$, $\frac{\partial \mc B_{\mc H,p}}{\partial x_v} = 0$. This is equivalent to \eqref{eq:BP fixed point definition}. The fact that $\mc B_{\mc H,p}$ attains its maximum follows from the fact that it is a continuous function on a compact set.

    To prove that $\mc B_{\mc H,p}$ attains its maximum on a BP fixed point, it suffices to prove that it does not attain its maximum on the boundary of $[0,1]^V$, and to observe that it is continuously differentiable on $(0,1)^V$. Let $v \in V$. To see that $\mc B_{\mc H,p}$ is not maximized by any $x$ in which $x_v=0$, we observe that $\frac{\partial \mc B_{\mc H,p}}{\partial x_v} > 0$ for all sufficiently small $x_v$. To see that $\mc B_{\mc H,p}$ is not maximized by any $x$ in which $x_v=1$, we observe that $\frac{\partial\mc B}{\partial x_v}\big|_{x_v=1}\leq\log p<0$.
    
    For the final observation, we note that if $x$ is a BP fixed point then for every $v$, $x_v = p \exp \left(- \sum_{e \in E_v} \prod_{u \in e} x_u \right) \leq p$. This also implies that $x_v = p \exp \left(-\sum_{e \in E_v} \prod_{u\in e}x_u\right) \geq p \exp \left(-|E_v|p^{k-1}\right) \geq p\exp(-\Delta p^{k-1}) = \Omega(p)$ by \cref{itm:setup p is critical}.
\end{proof}

\begin{definition}\label{def:tilde B}
    Define $\widetilde{\mathcal{B}}_{\mc H,p} : [0,1]^V \to \R$ by
    \[
    \widetilde{\mathcal{B}}_{\mc H,p} (x) = \sum_{v \in V} \left( x_v - p + \left( 1 - \frac{1}{k} \right) x_v \log \left( \frac{p}{x_v} \right) \right).
    \]
\end{definition}

\begin{observation}\label{obs:Bethe equivalence at BPFP}
    Let $x\in [0,1]^V$ be a BP fixed point with respect to $\mc H$ and $p$. Then
    \[
    \mc B_{\mc H,p}(x) = \widetilde{\mathcal{B}}_{\mc H,p}(x).
    \]
\end{observation}

\begin{proof}
Since $x$ is a BP fixed point, we have
\[
\log(p/x_v)
=
\sum_{e \in E_v} \prod_{u \in e} x_u\, .
\]
Multiplying by $x_v$ and summing over $v\in V$ yields
\[
\sum_{v\in V}x_v\log(p/x_v)
=
\sum_{v\in V}x_v\sum_{e \in E_v} \prod_{u \in e} x_u
=
k\sum_{e\in E} \prod_{u\in e}x_u
\]
where for the final equality we note that each edge is counted exactly $k$ times in the double sum. Substituting into the definition of $\mc B_{\mc H,p}(x)$ completes the proof.
\end{proof}

\section{Overview of the proof}\label{sec:motivation}

In this section we give an overview of the proof of our main theorem, Theorem~\ref{thm:Bethe exactness}.
By Observation~\ref{obs:hardcore partition and log probability} it suffices to find asymptotics for $\log Z_{\cH}(\lambda)$ where $\lambda=p/(1-p)$.
We use separate arguments to upper bound and lower bound $\log Z_{\cH}(\lambda)$. The starting point of both arguments is the Gibbs variational formula (Observation~\ref{obs:Gibbs})
 \[
 \log Z_{\cH}(\lambda)
 =
 \max_{\nu}
 \left\{
 H(\nu)+\log\lambda\,
 \mathbb E_{\nu}|\bb I|
 \right\},
 \]
where the maximum is over all probability distributions $\nu$ supported on the independent sets of $\cH$ and the maximum is attained at $\nu=\mu_{\mc H, p}$ where $\lambda = p/(1-p)$. For the upper bound, we therefore focus our attention on bounding the entropy $H(\mu_{\mc H,p})$. For the lower bound, it suffices to construct a distribution $\nu$ on independent sets that makes $H(\nu)+\log\lambda\,
\mathbb E_{\nu}|\bb I|$ large. In fact, we construct a family of distributions, one for each BP fixed point with respect to $\cH, p$: given a fixed point $x$, we consider a random greedy independent set process that is `guided by' $x$. We now turn to a more detailed discussion of these arguments.

\subsection{Upper bound}\label{sec:upper-bound}
Let $\bb I\sim \mu_{\cH,p}$.
By the equality case of the Gibbs variational formula (Observation~\ref{obs:Gibbs}) we have
\begin{equation}\label{eqn:EnergyEntropy}
\log Z_{\mc H}(\lam) = H(\mathbb{I}) + \log \lam \cdot \E[|\mathbb{I}|]\, .
\end{equation}
We upper-bound $\log Z_{\mc H}(\lam)$ via an upper-bound on the entropy $H(\mathbb{I})$. The proof is inspired by Radhakrishnan's method, first used to prove permanent inequalities \cite{radhakrishnan1997entropy}: we consider the vertices of $V$ in a random order, and successively reveal whether or not they are in $\mb I$. By bounding the entropy of each step in this process we bound the entropy of $\mb I$.

When bounding the entropy of the step where we reveal the state of $v$, it is useful to know something about the local environment of $v$. For example, if $v$ is blocked by previously revealed vertices, then the entropy of this step is $0$. This motivates the following definitions.   

A key component of our analysis is understanding when vertices can be added to a (usually random) independent set without forming an edge. To this end, for a set $S \subseteq V$, we define $\mc U(S) \coloneqq \{ v \in V : S \cup \{v\} \in \mc I(\mc H) \}$. We call these vertices \textit{unblocked}.

For $v \in V$, let $Y_v(I) \subseteq E_v$ denote the set of completions of $v$ that are fully unblocked by $I$. That is,
\[
Y_v(I) \coloneqq \{e \in E_v : e \subseteq \mc U(I)\}.
\]
Observe that it is not necessarily the case that $e \in Y_v(I) \implies I \cup e \in \mc I(\mc H)$. However, we will show that this is the typical situation. We write
\begin{align}\label{eq:empirical-marginals}
y_v(I) \coloneqq |Y_v(I)|, \quad x_v(I) \coloneqq p \exp \left( - p^{k-1}y_v(I) \right), \quad\text{and } \bar x(I) \coloneqq (x_v(I))_{v \in V}.
\end{align}
When $\bI \sim \mu_{\mc H,p}$, we  denote $\mathbf{Y}_v = Y_v(\mathbb{I})$, $ \mathbf{y}_v = y_v(\mathbb{I})$, $ \mathbf{x}_v = x_v(\mathbb{I})$, and $ \ol{\mathbf{x}} = \ol{x}(\mathbb{I})$. We will show that $\P(v\in \bb I)=\E[\mathbf x_v]+o(p)$ (see \cref{remark:empirical} below) and so we refer to $\mathbf x_v$ as the \emph{empirical marginal} at $v$.

To further motivate these definitions, we introduce our guiding heuristic: typically, $\mb I$ resembles a $p$-random subset of $\mc U(\mb I)$. Indeed, by Observation~\ref{obs:spatial-markov}, conditioned on $v \in \mc U(\mb I)$, the probability that $v \in \mb I$ is exactly $p$. Moreover, we will show that in general, for a small vertex set $U$, the events that each of the vertices in $U$ are unblocked are approximately independent (in an averaged sense, see~Lemma~\ref{lem:availability product decomposition} below).

As hinted above, it is crucial to understand the probability that a vertex $v$ is unblocked, (usually after conditioning on a subset of $\mb I$). It turns out that this probability is encoded by $\mathbf y_v$: in a $p$-random subset of $\mc U(\mb I)$, we would expect the number of edges blocking $v$ to be approximately $ p^{k-1} \mathbf {y}_v$. Thus, adopting a Poissonian heuristic, we expect that $\Pr[v \in \mc U(\mb I)] \approx \E \left[ \exp \left( - p^{k-1} \mathbf{y}_v \right) \right]$. This would imply that $\Pr[v \in \mb I] \approx \E \left[ p \exp \left( - p^{k-1} \mathbf{y}_v \right) \right] = \E[\mathbf x_v]$.

Of course, we cannot expect that $\mathbf y_v$ is concentrated around any particular value. For instance, in the case where $\cH$ encodes triangles in $K_n$, in the regime in which a typical sample is near-bipartite, $\mathbf y_v$ is quite small if $v$ corresponds to a ``crossing'' edge, while it is quite large if it corresponds to an ``internal'' edge. However, these events can be distinguished from a small amount of information. (In the triangle-free graph example, simply the bipartition itself. This can usually be reconstructed by randomly revealing a small number of edges.) In the proof, we will choose a sparse ``window'' $W \subseteq V$, and condition $\mb I$ on $\mb I \cap W$. As we will show, the information encoded by $W$ and $\mb I \cap W$ is enough to approximately reconstruct $\ol{\mathbf x}$.

This idea of inducing approximate independence by revealing a small fraction of information echoes the conditional entropy technique of Kozma and Samotij~\cite{kozma2023lower} in analyzing the mean-field approximation in the supercritical regime for the lower-tail problem as well as the `pinning lemma' from information theory~\cite{andrea2008estimating,allen2015conditioning,el2022information}, the study of approximation algorithms~\cite{barak2011rounding,raghavendra2012approximating}, mean-field approximations~\cite{jain2019mean,eldan2020taming} and rigorous implementations of the cavity method~\cite{coja2018information,coja2019spin}. 

One key way in which we make these heuristics precise is as follows. Recall from Observation~\ref{obs:Bethe equivalence at BPFP} that if $x$ is a BP fixed point with respect to $\mc H$ and $p$, then
    $
    \mc B_{\mc H,p}(x) = \widetilde{\mathcal{B}}_{\mc H,p}(x).
    $
We will show that $\ol{\mathbf x}$ is an approximate fixed point in the following sense.

\begin{lemma}\label{lem:fixedpointrewrite}
We have 
\[
\E[\mathcal{B}_{\mc H,p}(\ol{\mathbf{x}})] = \E[\widetilde{\mathcal{B}}_{\mc H,p}(\overline{\mathbf{x}})] + o(np).
\]
\end{lemma}
With this in hand, our goal will be to prove the following bound on the partition function.
\begin{lemma}\label{lem:partition function UB} 
    We have
    \begin{equation}
    \log Z_{\mc H}(\lam) \leq \E\left[\widetilde{\mathcal{B}}_{\mc H,p}(\bar {\mathbf{x}})\right] + np + o(np).
    \end{equation}
\end{lemma}

We now describe the `window' and the entropy argument in more detail.

\subsection{Entropy and the window}\label{sec:entropy-window}
The main insight is that $\mathbf y_v$ can be recovered by studying a small random subset of $\mb I$. More precisely, let $q=q(n)\in(0,1)$ tend to zero sufficiently slowly (as specified later in the proof of Lemma~\ref{lem:y-estimator})  and let $W$ be a $q$-random binomial subset of $V$ (which we call the \textit{window}). $W$ is chosen independently of $\mb I$. We will show that conditioning on $W \cap \mb I$, the typical value of $\mathbf y_v(\mb I)$ is determined up to a small error. Indeed, letting $\hat{\mathbf y}_v=\E[\mathbf y_v \mid W\cap\mathbb I]$ we show the following.
\begin{lemma}\label{lem:y-estimator}
Uniformly in $v\in V$,
\[
p^{k-1}
\E|{\mathbf y}_v-\hat{\mathbf y}_v|
=
o(1).
\]
\end{lemma}

Moreover, we will show that the combined entropy and energy cost of revealing $W \cap \mb I$ is small (see Lemma~\ref{lem:cheap-window} below). The window therefore gives us a cheap way to learn the `local environment' of a vertex.

We now describe the entropy argument. Let $\vec{\alpha} = (\alpha_v)_{v\in V} \in [0,1]^V$ be i.i.d.\ uniformly random values in $[0,1]$. For $v\in V$, let $X_v$ indicate the event $v \in \mb I$ and define $<_v := \{u \in V~|~\alpha_u < \alpha_v\}$. For $S\subseteq V$, let $X_S := \{X_v\}_{v\in S}$. Recalling that $W$ is independent of $\mb I$, and that both are independent of $\vec\alpha$, the entropy chain rule implies
\begin{align}
H(\mathbb{I}) = H(\mb I \mid W) &= H(\mathbb{I} \cap W \mid W) + H(\mathbb{I}~|~W \cap \mb I , W) \nonumber \\
&= H(\mathbb{I} \cap W \mid W) + H(\mathbb{I}~|~ W\cap \mathbb{I}, W,\vec \alpha) \nonumber \\
&= H(\mathbb{I} \cap W \mid W) + \sum_{v \in V}H(X_v~|~W\cap \mathbb{I}, W,\vec \alpha,X_{<_v})\, .\label{eqn:startpoint}
\end{align}

It is simple to show that $H(\mathbb{I} \cap W \mid W) + \log \lambda \E |\mb I \cap W|$ is small and so we fix $v\in V$ and turn our attention to bounding $H(X_v~|~W\cap \mathbb{I}, W,\vec \alpha,X_{<_v})$.
For convenience, let
$\cW$ denote the pair $(W\cap \mathbb{I}, W)$.

For $t\in [0,1]$, let $\mathbb I_t=\{u\in \mathbb I: \alpha_u<t\}$ and let 
\[
U_v(t)=\{v\in \mc U(\mathbb I_t) \}.
\]
In words, $U_v(t)$ denotes the event that $v$ is unblocked by the vertices revealed by time $t$. We also define
\[
u_v(t)
=
\P(U_v(t) \mid \cW)
\]
as the (random) probability that $U_v(t)$ occurs, conditioned on $\mc W$.

\subsubsection{Heuristic entropy bound}
We consider an outcome of the window with $v\notin W$, otherwise conditioning on $\cW$ reveals $X_v$.
Lemma~\ref{lem:y-estimator} shows that the window gives an estimate $\hat{\mathbf{y}}_v$ for $\mathbf y_v$. Heuristically, there are $\hat{\mathbf{y}}_v$ fully unblocked $e\in E_v$, each of which has probability approximately $(pt)^{k-1}$ of being occupied and revealed before time $t$. A Poisson approximation therefore suggests that
\begin{align}\label{eq:Poisson-heuristic-UB}
u_v(t)\approx e^{-(pt)^{k-1}\hat{\mathbf y}_v}.
\end{align}
On the other hand, by the spatial Markov property,
\[
\P(X_v=1\mid \cW)=p u_v(1),
\]
and the conditional probability of $X_v=1$, given that it is unblocked when it is revealed, is $p u_v(1)/u_v(t)$.
If $v$ is blocked when it is revealed, its entropy contribution is zero. Otherwise the contribution is at most  
$h(pu_v(1)/u_v(t))$ where $h(x) = - x \log x - (1-x)\log(1-x)$ denotes the entropy of a $\text{Ber}(x)$ random variable. The total entropy is therefore at most
\[
u_v(t)h\left(\frac{p u_v(1)}{u_v(t)}\right)\, .
\]
Combining this with the corresponding
energy contribution from $v$ (the contribution to $\log\lambda\,
 \mathbb E |\bb I|$) we obtain
heuristically
\[
u_v(t)h\left(\frac{p u_v(1)}{u_v(t)}\right)
+
p u_v(1)\log \lambda
\approx
pe^{-p^{k-1}\hat{\mathbf{y}}_v}\bigl(1+(1-t^{k-1})p^{k-1}\hat{\mathbf{y}}_v\bigr).
\]
Returning to the Gibbs variational formula~\eqref{eqn:EnergyEntropy}, we integrate the above estimate over $t$, sum over $v$, and average over the window to conclude that
\begin{align*}
\log Z_{\mathcal H}(\lambda)
\lesssim 
\E\sum_{v\in V}
pe^{-p^{k-1}\hat{\mathbf{y}}_v}\left(1+\frac{k-1}{k}p^{k-1} \hat{\mathbf{y}}_v\right)
&\approx
\E\sum_{v\in V}
pe^{-p^{k-1} \mathbf y_v}\left(1+\frac{k-1}{k}p^{k-1}  \mathbf y_v\right)\\
&\approx
\E[\widetilde{\mathcal B}_{\mathcal H,p}
(\overline{\mathbf x})]+np
\end{align*}
where in the second approximation we used that $\hat{\mathbf{y}}_v$ is a good estimator for $\mathbf y_v$ (Lemma~\ref{lem:y-estimator}). Thus we (heuristically) arrive at Lemma~\ref{lem:partition function UB}. Our desired upper bound follows from Lemma~\ref{lem:fixedpointrewrite} and the bound $\E[\mathcal{B}_{\mc H,p}(\ol{\mathbf{x}})] \leq \max_{x\in[0,1]^V}\mathcal{B}_{\mc H,p}(x)$. The main work in making this rigorous is to prove~\cref{lem:y-estimator} and to make rigorous the Poisson heuristic at~\eqref{eq:Poisson-heuristic-UB}.

\subsection{Lower bound}\label{sec:overviewLB}
Recall that to prove a lower bound on $\log Z_{\cH}(\lambda)$, it is enough to construct the distribution of one random independent set $I$ and lower-bound
 $
 H(I)+\log\lambda\,\mathbb E|I|.
 $
 To this end, we will use a simple random greedy algorithm for constructing an independent set $I$. We describe the algorithm and then motivate it. 

 First we fix a BP fixed point $x\in [0,1]^V$. For $t\in[0,1]$ and $v\in V$ let 
 \begin{equation}\label{eq:r_v definition}
 r_v(t)= p (x_v/p)^{1-t^{k-1}}\, .
 \end{equation}

\subsubsection{The BP-guided random greedy process on $\cH$.}\label{subsec:BP-guided} Start with $I=\emptyset$. Independently for every $v\in V$, choose
 \begin{align}\label{eq:tau-xi}
 \tau_v,\xi_v\sim U[0,1].
 \end{align}
 We refer to $\tau_v, \xi_v$ as the `clock' and `coin' at $v$ respectively. We refer to the pair $(\tau_v, \xi_v)$ as the `mark' at $v$. We process the vertices in increasing order of $\tau_v$. When the clock of $v$ rings at time $t=\tau_v$, add $v$ to $I$ if both of the following hold:
 \begin{enumerate}
 \item no $e\in E_v$ is already contained in $I$,
\item $\xi_v\leq r_v(t)$.
\end{enumerate}
The first condition says that $v$ is unblocked. The second ensures that if $v$ is unblocked, we include it in the independent set with probability $r_v(t)$.

\subsubsection{Heuristic analysis of the process}\label{subsec:heuristic}
Let $I_{<t}$ denote the set of vertices occupied (included in the independent set) before time $t$, and define
\begin{align}\label{eq:av-Def}
a_v(t)
=
\P\bigl(v\in\mc U(I_{<t})\mid \tau_v=t\bigr)\, .
\end{align}
In words, $a_v(t)$ is the probability that $v$ is unblocked when its clock rings at time $t$.

Imagine for a moment that, by time $t\in[0,1]$, each vertex $u$ has been occupied independently with probability $tx_u$.  
The expected number of $e\in E_v$ that block $v$ at time $t$ would then be
\[
t^{k-1}\sum_{e\in E_v}\prod_{u\in e}x_u\, .
\]
 A Poisson heuristic then suggests that
\begin{align}\label{eq:av-estimate}
a_v(t)
\approx
\exp\left(-t^{k-1}\sum_{e\in E_v}\prod_{u\in e}x_u \right)
=
(x_v/p)^{t^{k-1}}.
\end{align}
where for the final equality we used that $x$ is a BP fixed point. We can now see the motivation for the definition of $r_v(t)$ above. It ensures that when the clock at $v$ rings at time $t$, the probability that $v$ is included in $I$ is 
\[
a_v(t) r_v(t)
\approx
(x_v/p)^{t^{k-1}} \cdot r_v(t)
= x_v\, ,
\]
independent of $t$.

We now show this (heuristically) gives the correct entropy.
Recall that $h(q)$ denotes the entropy of a $\text{Ber}(q)$ random variable and recall that $h(q)=q (\log(1/q)+1) + O(q^2)$. 

Suppose the clock at $v$ rings at time $t$. If it is blocked the indicator that $v$ is occupied has entropy $0$. If it is unblocked, then this indicator has entropy $h(r_v(t))$. We therefore have
\begin{align*}
H( I \mid (\tau_v)_{v\in V})
=
\sum_{v\in V} \int_{0}^1 
a_v(t)
h(r_v(t))\, dt
&\approx
\sum_{v\in V} \int_{0}^1 
x_v
\left(
\log(1/r_v(t)) + 1
\right)
\, dt \\
&=
\sum_{v\in V}
x_v 
\left(
\log(1/p)+(1-1/k)\log (p/x_v) + 1
\right)\, .
\end{align*}
Since conditioning can only decrease entropy we conclude that
\begin{align*}
H(I) + \log \lambda \E |I|
&\geq
\sum_{v\in V}
x_v 
\left(
\log(1/p)+(1-1/k)\log (p/x_v) + 1
\right)
+
\log p
\sum_{v\in V} x_v\\
&=
\widetilde{\mathcal{B}}_{\mc H,p} (x) +np\, .
\end{align*}
To turn this heuristic into a proof, the main work is in rigorously justifying the approximation~\eqref{eq:av-estimate} (see Lemma~\ref{lem:root-availability} below).

\section{The upper bound in \cref{thm:Bethe exactness}}\label{sec:upper bound}
\newcommand{\qvalue}{o(1)}

In this section we prove the upper bound in \cref{thm:Bethe exactness}. Recalling~Observation~\ref{obs:hardcore partition and log probability}, this is the content of the following lemma.
\begin{lemma}\label{lem:Bethe FE upper bound}
    In the setup of \cref{thm:Bethe exactness}, with $\lambda=p/(1-p)$, it holds that
    \begin{equation}
    \log Z_{\mc H}(\lam) \leq \max_{x\in [0,1]^V} \mathcal{B}_{\mathcal{H},p}(x) + np + o(np).
    \end{equation}
\end{lemma}

First we recall the overview from Section~\ref{sec:upper-bound}, and note that Lemma~\ref{lem:Bethe FE upper bound} will follow immediately once we prove Lemmas~\ref{lem:partition function UB} and~\ref{lem:fixedpointrewrite} since $\E\left[{\mathcal{B}}_{\mc H,p}(\bar {\mathbf{x}})\right]\leq \max_{x\in [0,1]^V} {\mathcal{B}}_{\mc H,p}(x) $.

Returning to~\eqref{eqn:startpoint},
let us first argue that $H(\mathbb{I} \cap W \mid W) + \log \lambda\cdot \E [|\mathbb{I} \cap W|]$ is small. In view of~\eqref{eqn:EnergyEntropy}, we prove the following.  

\begin{lemma}\label{lem:cheap-window}
\[
H(\mathbb{I} \cap W \mid W)
+
\log \lambda\cdot 
\E [|\mathbb{I} \cap W|]
\leq
o(np)\, .
\]
\end{lemma}
\begin{proof}
Let
$
h(z) = -z\log(z) - (1-z)\log (1-z)
$
denote the Shannon entropy of a $\text{Ber}(z)$ random variable.
For every $S \subseteq V$, we have
\begin{align}
H(\mathbb{I} \cap W \mid W = S) 
+ \log \lambda \cdot
\E [|\mathbb{I} \cap W| \mid W=S]\nonumber
&\leq 
\sum_{v\in S} \left(h(\E[X_v])
+
\log \lambda\cdot  \E[X_v]\right)\\
&\leq 
\log(1+\lambda) |S|\, .
\end{align}
The first inequality uses subadditivity of entropy and independence of $W$ and $\bb I$, and the second inequality uses that $h(z)+ z\log \lambda\leq \log(1+\lambda)$ for all $z\in (0,1)$. Taking an expectation over $W$ and noting that $\log(1+\lambda)=O(p)$ and $\E |W|= nq=o(n)$
completes the proof.
\end{proof}

We now fix $v\in V$ and turn our attention to bounding $H(X_v~|~\cW,\vec \alpha,X_{<_v})$. Recall that $\vec{\alpha} = (\alpha_v)_{v\in V} \in [0,1]^V$ where the $\alpha_v$ are i.i.d.\ uniformly random values in $[0,1]$. 
Recall also that $t\in [0,1]$, $\mathbb I_t=\{u\in \mathbb I: \alpha_u<t\}$,  
$
U_v(t)=\{v\in \mc U(\mathbb I_t) \},
$
and
$
u_v(t)
=
\P(U_v(t) \mid \cW)\, .
$
We note that the event $U_v(t)$ does not depend on $\alpha_v$. Hence
$
u_v(t)=\P(U_v(t)\mid\cW,\alpha_v=t).
$
\begin{lemma}\label{lem:Vx-entropy-bd}
\[
H(X_v~|~\cW,\vec \alpha,X_{<_v})
\leq 
\E
\left[
\mathbf 1_{\{v\notin W\}}
\int_{0}^1
u_v(t)
h\left(
\frac{p u_v(1)}{u_v(t)}
\right)\, dt\, 
\right]\, .
\]
\end{lemma}
\begin{proof}
Let $\cW_{S,T}$ denote the event that $(W\cap \mathbb{I}, W)=(S,T)$. Fixing a pair $S,T$ with $\P(\cW_{S,T})>0$, we look to bound $H(X_v~|~\cW_{S,T},\vec \alpha,X_{<_v})$. We may assume that $v\notin T$ else the entropy is $0$. 

It will be convenient to condition on the value of $\alpha_v$. For notational convenience, let us define the event $\cE_t= \cW_{S,T}\cap \{\alpha_v=t\}$. 
Let 
\[
u^{S,T}_v(t)
=
\P(U_v(t) \mid \cE_t)
\]
In particular $u^{S,T}_v(1)=\P(v\in \mc U(\mathbb I) \mid \cW_{S,T})$ and
\[
\P(X_v=1 \mid  \cE_t)
=
p u^{S,T}_v(1)
\]
where we used that $\vec \alpha$ is independent of $(\bb I, W)$, $v\notin T$, and Observation~\ref{obs:spatial-markov} (applied with $R=\{v\}$ and $F(I)=\mathbf 1_{\{I\cap T=S\}}$).

Next note that, conditional on $\alpha_v=t$, $\vec \alpha$ and $X_{<v}$ determine $U_v(t)$ and so 
\begin{align}\label{eq:entropyWST}
H(X_v~|~\cE_t,\vec \alpha,X_{<_v})
\leq 
H(X_v~|~\cE_t, \mathbf 1_{U_v(t)})
=
u^{S,T}_v(t)
h\left(
\frac{p u^{S,T}_v(1)}{u^{S,T}_v(t)}
\right)\, ,
\end{align}
where for the final equality we used that $X_v=0$ whenever $U_v(t)$ fails to hold and 
\[
\P(X_v=1 \mid U_v(t), \cE_t)
=
\frac{\P(X_v=1 \mid \cE_t)}
{\P(U_v(t) \mid \cE_t)}
=
\frac{p u^{S,T}_v(1)}{u^{S,T}_v(t)}\, .
\]
If $u^{S,T}_v(t)=0$, we interpret the RHS of~\eqref{eq:entropyWST} as $0$. Integrating~\eqref{eq:entropyWST} over $t$ and taking an expectation with respect to $\cW$ yields the result. 
\end{proof}

With this lemma in hand, we turn our attention to approximating the availability probability $u_v(t)$. 
The following lemma makes the Poisson heuristic at~\eqref{eq:Poisson-heuristic-UB} formal.

\begin{lemma}\label{lem:availability-target}
We have
\[
\E \left|
u_v(t)
-
e^{-(pt)^{k-1}{\mathbf{y}}_v}
\right|
=o(1)
\]
uniformly in $v\in V$ and $t\in [0,1]$.
\end{lemma}

Before turning to the proof this lemma, we show how it is used to complete the proof of Lemma~\ref{lem:partition function UB}.

\begin{proof}[Proof of Lemma~\ref{lem:partition function UB}]
We begin by observing that
\[
\E|\mathbb I \setminus W| 
=
p\sum_{v\in V}\E[\mathbf 1_{\{v\notin W\}}u_v(1)]\, .
\]
Indeed, on the event $\{v\notin W\}$, Observation~\ref{obs:spatial-markov} gives 
\begin{align}\label{eq:occ-cond-window}
\P(v\in\mathbb I\mid\cW)=p u_v(1)\, .
\end{align}
Combining this with~\eqref{eqn:EnergyEntropy},~\eqref{eqn:startpoint}, Lemma~\ref{lem:cheap-window} and Lemma~\ref{lem:Vx-entropy-bd} shows that
\[
\log Z_{\cH}(\lambda)
\leq
\sum_{v\in V}
\E
\left[
\mathbf 1_{\{v\notin W\}}
\int_{0}^1
u_v(t)
h\left(
\frac{p u_v(1)}{u_v(t)}
\right) 
+ pu_v(1)\log \lambda  \, dt
\right]  
+o(np)\, .
\]
Using the bound $h(z)\leq z(\log(1/z)+1)$ for $z\in[0,1]$ 
and noting that $\log(\lambda/p)=O(p)$
yields
\[
\log Z_{\cH}(\lambda)
\leq
p\sum_{v\in V}
\E
\left[
\mathbf 1_{\{v\notin W\}}
\int_{0}^1
 u_v(1)\left(\log\left( \frac{u_v(t)}{ u_v(1)}\right) 
+
1
\right)
\, dt
\right]  
+o(np)\, .
\]
Since the function $(x,y)\mapsto x (\log(y/x)+1)$ is bounded, non-negative, and uniformly continuous on the set 
$ \{(x,y): 0\leq x\leq y\leq 1\}$ (defining $x (\log(y/x)+1)=0$ when $x=0$)
we may drop the indicator and apply Lemma~\ref{lem:availability-target} to conclude that 
\begin{align*}
\log Z_{\cH}(\lambda)
&\leq
p\sum_{v\in V}
\E
\left[
\int_{0}^1
 e^{-p^{k-1}{\mathbf{y}}_v}
 \left(
 (1-t^{k-1})p^{k-1}{\mathbf{y}}_v
+
1
\right)
\, dt
\right]  
+o(np)\\
&
= 
p\sum_{v\in V}
\E
\left[
 e^{-p^{k-1}{\mathbf{y}}_v}
 \left(
\frac{k-1}{k}p^{k-1}{\mathbf{y}}_v
+
1
\right)
\right]  
+o(np)\\
&=
 \E\left[\widetilde{\mathcal{B}}_{\mc H,p}(\bar {\mathbf{x}})\right] + np + o(np)\, ,
\end{align*}
where for the final equality we recalled Definition~\ref{def:tilde B} (the definition of $\widetilde{\mathcal{B}}_{\mc H,p}$) and recalled that $\mathbf x_v=pe^{-p^{k-1}\mathbf y_v}$\, .
\end{proof}

To complete the proof of Lemma~\ref{lem:Bethe FE upper bound}, it remains to prove Lemma~\ref{lem:fixedpointrewrite}
and Lemma~\ref{lem:availability-target}. A key step along the way will be to prove Lemma~\ref{lem:y-estimator} which we turn to now. 

\subsection{Estimators from the window}

In this section we prove Lemma~\ref{lem:y-estimator}. Our goal is to find a $W\cap \mathbb{I}$-measurable random variable that provides a good estimate for $\mathbf y_v$. Our starting point is the following basic observation. If $S\subseteq V$, then a natural $W\cap \mathbb{I}$-measurable estimator for $\mathbf 1_{\{S\subseteq \mathbb I\}}$ is $q^{-|S|}\mathbf 1_{\{S\subseteq \mathbb I \cap W\}}$. Observe that it is unbiased with respect to $W$: for any fixed $I$
\begin{align}\label{eq:unbiased}
\E_W[q^{-|S|}\mathbf 1_{\{S\subseteq  I \cap W\}}] = \mathbf 1_{\{S\subseteq  I\}}\, .
\end{align}
Recall that $\mathbf y_v=\sum_{e\in E_v}\mathbf 1_{\{e\subseteq \mc U(\mathbb I)\}}$ which is notably not of the above form. However, $\mathbf y_v$ is naturally expressed as a linear combination of indicators of the form $1_{\{S\subseteq \mathbb I\}}$ via inclusion-exclusion. For this we need a bit of notation. For a set $S\subset V$, we let
\begin{align}\label{eq:cTdef}
\mc T(S) 
= 
\bigcup_{u\in S}E_u
\end{align}
and for $F\subseteq \mc T(S)$, let $V(F)=\bigcup_{f\in F}f$. By inclusion-exclusion we then have
\[
\mathbf 1_{\{e\subseteq \mc U(\mathbb I)\}} 
=
\sum_{j\geq 0} (-1)^j
\sum_{F \in \binom{\cT(e)}{j}} 
\mathbf 1_{\{V(F) \subseteq \mathbb I\}}\, .
\]
This suggests the following $(W\cap \mathbb{I})$-measurable estimator for $\mathbf y_v$:
\[
\sum_{e\in E_v}\sum_{j\geq 0} (-1)^j
\sum_{F \in \binom{\cT(e)}{j}}
q^{-|V(F)|}
\mathbf 1_{\{V(F) \subseteq \mathbb I \cap W\}}\, .
\]
In practice, we will truncate the alternating sum as this will allow us to more easily control fluctuations of the estimator. 

Before proceeding to the proof of Lemma~\ref{lem:y-estimator}, we will need the following bounds which follow from the fact that $\cH$ is $\eta$-locally tree-like (Definition~\ref{def:LTL}).

\begin{lemma}\label{lem:tree-estimator}
There is a constant $K$ and constants $C_j, C_{a,b}$ such that the following hold. For $S\subseteq V$ with $|S|\leq k$ and $j\geq 1$,
\begin{align}\label{eq:tree-estimator1}
\sum_{F\in \binom{\cT(S)}{j}} p^{|V(F)|}
\leq
\frac{K^j}{j!} + C_j \eta
\end{align}
and for $a,b\geq 1$ and $v\in V$
\begin{align}\label{eq:tree-estimator2}
\sum_{e,f\in E_v}
\sum_{\substack{F\in\binom{\mc T(e)}{a}, F'\in\binom{\mc T(f)}{b}\\
V(F)\cap V(F')\neq\emptyset}}
p^{|V(F)\cup V(F')|}
\leq
C_{a,b}\eta \Delta^2\, .
\end{align}
\end{lemma}
\begin{proof}
We begin with~\eqref{eq:tree-estimator1}. Fix a sequence of integers $\vec \ell=(\ell_1, \ldots, \ell_j)$ where $0\leq \ell_i \leq k-1$ for all $i$. 
We consider the contribution to the sum~\eqref{eq:tree-estimator1} from those $F$ whose elements admit an ordering $(f_1, \ldots, f_j)$ such that $|f_i\backslash (f_1\cup\ldots\cup f_{i-1})|=\ell_i$ for $1\leq i \leq j$. Let $\cT_{\vec \ell}$ denote the collection of all such tuples $(f_1, \ldots, f_j)\in \cT(S)^j$. Note that if $F=\{f_1, \ldots, f_j\}$ and $(f_1, \ldots, f_j)\in \cT_{\vec \ell}$ 
 then $|V(F)|=\sum_{i=1}^j \ell_i$.

First we bound the size of $\cT_{\vec \ell}$. We build a tuple $(f_1, \ldots, f_j)\in \cT_{\vec \ell}$ iteratively. 
Note that given $f_{\leq i-1}:=f_1\cup\ldots\cup f_{i-1}$, there are 
\[
\binom{|f_{\leq i-1}|}{k-1-\ell_i}|S| \Delta_{k-\ell_i}
\]
choices for $f_{i}$. Indeed, the binomial coefficient counts the number of choices for $f_{i}\cap f_{\leq i-1}$, there are $|S|$ choices for the $u\in S$ for which $f_{i}\in E_u$, and there are then $\Delta_{k-\ell_i}$ choices for the remaining vertices of $f_{i}$.
Since $|f_{\leq i-1}|\leq ki$, we therefore have
\begin{align}\label{eq:contribution-ell}
|\cT_{\vec \ell}|\leq \prod_{i=1}^{j} \binom{ki}{k-1-\ell_i}|S|\Delta_{k-\ell_i}\, .
\end{align}
We now distinguish two cases. First suppose that $F\in \binom{\cT(S)}{j}$ such that for all distinct pairs $e,f\in F$ we have $e\cap f=\emptyset$. In this case any ordering $(f_1,\ldots,f_j)$ of the elements of $F$ belongs to $\cT_{\vec \ell}$ where $\vec \ell=(k-1, \ldots, k-1)$. Using~\eqref{eq:contribution-ell}, we see that the contribution to the sum in~\eqref{eq:tree-estimator1} from such $F$ is at most
\[
\frac{1}{j!}\prod_{i=1}^{j} |S|\Delta p^{k-1} \leq \frac{K^j}{j!}
\]
for some constant $K>0$ where the factor of $1/j!$ accounts for all of the possible orderings of the elements $F$ and we recall that $|S|\leq k$ and $\Delta p^{k-1}=O(1)$.

Suppose instead that $F$ contains $e,f$ such that $e\cap f\neq \emptyset$ and $e\neq f$. In this case let $(f_1,\ldots, f_j)$ be an ordering of the elements of $F$ where $f_1=f, f_2=e$. In this case $(f_1,\ldots, f_j)\in \cT_{\vec \ell}$ for some $\vec \ell$ satisfying $1\leq\ell_2\leq k-2$. Again using \eqref{eq:contribution-ell}, we see that the contribution from such $F$ to the sum in ~\eqref{eq:tree-estimator1} is at most
\[
\sum_{\vec \ell:1\leq\ell_2\leq k-2}
 \prod_{i=1}^{j} \binom{ki}{k-1-\ell_i}|S|\Delta_{k-\ell_i}p^{\ell_i} \leq C_j \eta
\]
for some constant $C_j>0$. Indeed there are at most $k^j$ choices for $\vec\ell$, $\Delta_{k-\ell_i}p^{\ell_i}=O(1)$ for all $i$ and $\Delta_{k-\ell_2}p^{\ell_2}= O(\eta)$ since $1\leq\ell_2\leq k-2$. This establishes~\eqref{eq:tree-estimator1}.

We now prove~\eqref{eq:tree-estimator2}. Fix $e,f\in E_v$. If $F\in\binom{\mc T(e)}{a}, F'\in\binom{\mc T(f)}{b}$ such that 
$V(F)\cap V(F')\neq\emptyset$ then there exists $h\in F, g\in F'$ such that $h\cap g\neq \emptyset$. We again distinguish two cases. If $h\neq g$, then by the same sequential counting argument as above, the contribution to the inner sum of~\eqref{eq:tree-estimator2} from such $F, F'$ is at most $O_{a,b}(\eta)$.

Suppose then that $h=g$. In this case, there is $g\in E_u \cap E_w$ where $u\in e, w\in f$. If $u\neq w$ there are at most $|E_u \cap E_w|\leq \eta \Delta$ choices for $g$ by the codegree bound. If $u=w$ there are at most $|E_u|\leq \Delta$ choices for $g$. It follows that 
\begin{align}\label{eq:g-inter}
\sum_{g\in \cT(e)\cap\cT(f)}p^{|g|}\leq k^2\eta \Delta p^{k-1} + |e\cap f|\Delta p^{k-1} = O(\eta + |e\cap f|)
\end{align}
where the factor of $k^2$ accounts for the number of choices of $u\in e, w\in f$ with $u\neq w$ and $|e\cap f|$ accounts for the number of choices with $u=w$. Now given $g \in \cT(e)\cap\cT(f)$, we  consider all of the ways of choosing the remaining $a-1$ elements of $F$ and $b-1$ elements of $F'$. Arguing as in the proof of~\eqref{eq:tree-estimator1} we have
\[
\sum_{\substack{F\in\binom{\mc T(e)}{a}, F'\in\binom{\mc T(f)}{b}\\
g\in F\cap F'}}
p^{|V(F)\cup V(F')|-|g|}
=
O_{a,b}(1)\, .
\]
Combined with~\eqref{eq:g-inter} and our analysis of the case where $h\neq g$, we conclude that
\[
\sum_{\substack{F\in\binom{\mc T(e)}{a}, F'\in\binom{\mc T(f)}{b}\\
V(F)\cap V(F')\neq\emptyset}}
p^{|V(F)\cup V(F')|}
=
O_{a,b}(\eta + |e\cap f|)\, .
\]
Finally summing over $e,f\in E_v$ and noting that 
\[
\sum_{e,f\in E_v}|e\cap f|
=
\sum_{e\in E_v}\sum_{u\in e}
|\{f\in E_v:u\in f\}|
\leq (k-1)\Delta\Delta_2
=O(\eta\Delta^2),
\]
completes the proof.
\end{proof}

\begin{proof}[Proof of Lemma~\ref{lem:y-estimator}]
First note that it suffices to construct a $(W\cap \mathbb{I})$-measurable random variable $\widetilde{\mathbf y}_v$ satisfying 
$
p^{k-1}\E |\widetilde{\mathbf y}_v- {\mathbf y}_v|=o(1)
$
since, recalling that $\hat{\mathbf y}_v=\E[{\mathbf y}_v \mid W\cap \mathbb I]$, we then have
\begin{align*}
p^{k-1}\E |\hat{\mathbf y}_v- \widetilde{\mathbf y}_v|
=
p^{k-1}\E \left|
\E[{\mathbf y}_v-\widetilde {\mathbf y}_v\mid W\cap \mathbb I]
\right| 
\leq 
p^{k-1}\E |{\mathbf y}_v- \widetilde{\mathbf y}_v|
\end{align*}
To this end we fix an integer $L\geq 0$ and define,
\begin{align}\label{eq:trunc-estimator}
\widetilde{\mathbf y}_v
=
\sum_{e\in E_v}\sum_{j= 0}^L (-1)^j
\sum_{F \in \binom{\cT(e)}{j}}
q^{-|V(F)|}
\mathbf 1_{\{V(F) \subseteq \mathbb I \cap W\}}\, .
\end{align}
By the triangle inequality and Cauchy-Schwarz we have
\begin{align}
\E |\widetilde{\mathbf y}_v- {\mathbf y}_v|
&\leq
\E |\widetilde{\mathbf y}_v- \E_W[\widetilde{\mathbf y}_v \mid \mathbb I]|
+
\E |{\mathbf y}_v- \E_W[\widetilde{\mathbf y}_v \mid \mathbb I]|\nonumber\\
&\leq
(\E\var_W(\widetilde{\mathbf y}_v \mid \mathbb I))^{1/2}
+
\E |{\mathbf y}_v- \E_W[\widetilde{\mathbf y}_v \mid \mathbb I]|\label{eq:Var-W}
\end{align}
We bound these terms separately starting with the rightmost term. 
Note that by~\eqref{eq:unbiased} we have
\[
\E_W[\widetilde{\mathbf y}_v \mid \mathbb I]
=
\sum_{e\in E_v}\sum_{j= 0}^L (-1)^j
\sum_{F \in \binom{\cT(e)}{j}}
\mathbf 1_{\{V(F) \subseteq \mathbb I \}}\, .
\]
Using the Bonferroni inequality
\[
\left|
\sum_{j= 0}^L (-1)^j
\sum_{F \in \binom{\cT(e)}{j}}\mathbf 1_{\{V(F) \subseteq \mathbb I \}}
-
\mathbf 1_{\{e\subseteq \mc U(\mathbb I)\}}
\right|
\leq 
\sum_{F \in \binom{\cT(e)}{L+1}}\mathbf 1_{\{V(F) \subseteq \mathbb I \}}
\]
we conclude that
\begin{align*}
\E\left|
\mathbf y_v-
\E_W[\widetilde{\mathbf y}_v \mid \mathbb I]
\right|
\leq 
\E
\left[
\sum_{e\in E_v}
\sum_{F \in \binom{\cT(e)}{L+1}}\mathbf 1_{\{V(F) \subseteq \mathbb I \}}
\right]
&\leq 
\sum_{e\in E_v}
\sum_{F \in \binom{\cT(e)}{L+1}} p^{|V(F)|}\\
&\leq 
\left(\frac{K^{L+1}}{(L+1)!}
+
C_{L+1}\eta
\right)\Delta\\
&=o(\Delta)
\end{align*}
where for the second inequality we used \cref{obs:spatial-markov}, for the penultimate inequality we used~\eqref{eq:tree-estimator1}, and for the final inequality we recall that $\eta\to0$ and we take $L\to\infty$ sufficiently slowly. 

We now bound the variance term in~\eqref{eq:Var-W}. Returning to~\eqref{eq:trunc-estimator} note that $\var_W(\widetilde{\mathbf y}_v \mid \mathbb I)$ is a sum of covariances of the form
\begin{align*}
\textrm{Cov}_W
&\left(q^{-|V(F)|}\mathbf 1_{\{V(F) \subseteq \mathbb I \cap W\}},
q^{-|V(F')|}\mathbf 1_{\{V(F') \subseteq \mathbb I \cap W\}}
\mid \mathbb I\right)\\
&=
\left(
q^{-|V(F)\cap V(F')|}
-
1
\right)
\mathbf 1_{\{V(F)\cup V(F')\subseteq \mathbb I\}}
\end{align*}
(along with some sign) where for the equality we used~\eqref{eq:unbiased}. Note that this is non-zero only if $V(F)\cap V(F')\neq\emptyset$. Using \cref{obs:spatial-markov} and Lemma~\ref{lem:tree-estimator}, we conclude that
\begin{align*}
\E\var_W(\widetilde{\mathbf y}_v \mid \mathbb I)
&\leq
\sum_{e,f\in E_v}
\sum_{1\leq i,j\leq L}
\sum_{\substack{F\in\binom{\mc T(e)}{i}, F'\in\binom{\mc T(f)}{j}\\
V(F)\cap V(F')\neq\emptyset}}
q^{-|V(F)\cap V(F')|}
p^{|V(F)\cup V(F')|}\\
&\leq
\Delta^2q^{-(k-1)L}\eta \sum_{i,j=1}^LC_{i,j}
=o(\Delta^2)
\end{align*}
where the final inequality holds by taking $q\to0$ and $L\to \infty$ sufficiently slowly. The result now follows from~\eqref{eq:Var-W}, recalling that $p^{k-1}\Delta=O(1)$.
\end{proof}

We now turn our attention to proving~Lemma~\ref{lem:availability-target} which establishes a Poisson approximation for $u_v(t)$. To establish the Poisson approximation we use the method of Chen and Stein, first developed by Stein~\cite{stein1972bound} for normal approximations and then developed by Chen~\cite{chen1975poisson} in the Poisson setting. Our presentation of the method is inspired by that of~\cite{arratia1989two}.

\subsection{The Chen--Stein method for Poisson approximation}
In order to prove~Lemma~\ref{lem:availability-target}, we make use of the Stein equation 
\begin{align}\label{eq:stein-eqn}
t f(t,m+1)-m f(t,m)
=
\mathbf 1_{\{m=0\}}-e^{-t}
\end{align}
where 
\begin{align}\label{eq:stein-f}
f(t,m):=\int_0^1 s^{m-1}e^{-ts}\,ds\, ,
\end{align}
for $t\geq 0, m\geq 1$ (and we define $f(t,0)=0$).
The identity~\eqref{eq:stein-eqn} is easily verified using integration by parts. The precise form of $f$ given in~\eqref{eq:stein-f} is not important in our proofs, all we need are  the following elementary bounds:
\begin{align}\label{eq:stein-bounds}
\|f\|_\infty\le1,\qquad
|f(t,m)-f(t,m')|\le|m-m'|,\qquad
|f(t,m)-f(t',m)|\le |t-t'|.
\end{align}

Let $\cT\subseteq \binom{V}{k-1}$ be a family of $(k-1)$-sets. Let
\[
\xi_e:=\mathbf 1_{\{e\subseteq \mb I\}},\qquad
a_e:=\mathbf 1_{\{e\subseteq \mc U(\mb I)\}},\qquad
b_e:=\mathbf 1_{\{\mb I\cup e\in\mathcal I(\mc H)\}},
\]
and define
\begin{align}\label{eq:XY-def}
X:=\sum_{e\in\cT}\xi_e,\qquad
Y:=\sum_{e\in\cT}a_e.
\end{align}
Given a random variable $U=U(\mb I)$, let $U^e$ denote the random variable given by $U^e(\mb I)=U(\mb I\cup e)$ if $\mb I\cup e\in\cI(\cH)$ and $U^e(\mb I)=0$ otherwise.

The following lemma is the first step in establishing the Poisson heuristic.

\begin{lemma}\label{lem:stein-app}
Let $0\le G=G(\mb I)\le1$.  Then
\begin{align*}
\Big|
\mathbb E\left[G\mathbf 1_{\{X=0\}}\right]
&-
\mathbb E\left[Ge^{-p^{k-1}Y}\right]
\Big|\\
&\le
p^{k-1}\sum_{e\in\cT}
\E\left[
       |a_e-b_e|
       +b_e|G^e-G|
       + b_e|X^e-X-1|
       +p^{k-1}b_e|Y^e-Y|
\right].
\end{align*}
\end{lemma}

\begin{proof}
We multiply the Stein equation~\eqref{eq:stein-eqn} by $G$, set $t=p^{k-1}Y$, $m=X$, and take an expectation to obtain
\begin{align}\label{eq:stein-avg}
\E \left[G \mathbf 1_{\{X=0\}}\right] 
-
\E \left[Ge^{-p^{k-1}Y}\right]
=
\E \left[ Gp^{k-1} Y  f(p^{k-1}Y, X+1) \right]
-
\E \left[GX f(p^{k-1} Y, X)\right].
\end{align}
Recalling the definition of $Y$ at~\eqref{eq:XY-def} we have 
\[
\E \left[ Gp^{k-1} Y  f(p^{k-1}Y, X+1) \right]
=
p^{k-1} \sum_{e\in \cT}
\E \left[a_e G  f(p^{k-1}Y, X+1) \right]\, .
\]
For the second term in~\eqref{eq:stein-avg} we use the definition of $X$ and Observation~\ref{obs:spatial-markov} to obtain
\[
\E \left[GX f(p^{k-1} Y, X)\right]
=
\sum_{e\in \cT}
\E \left[G \xi_e f(p^{k-1} Y, X)\right]
=
p^{k-1}
\sum_{e\in \cT}
\E \left[b_e G^e  f(p^{k-1} Y^e, X^e)\right]
\]
Subtracting these two expressions, using the identity
\[
x_1 y_1 z_1-x_2y_2z_2
= 
y_1z_1(x_1-x_2)
+
x_2z_1(y_1-y_2)
+
x_2y_2(z_1-z_2)
\]
and the fact that $G, f$ are bounded by 1, we obtain
\begin{align*}
\Big|
\E \left[G \mathbf 1_{\{X=0\}}\right] 
&-
\E \left[Ge^{-p^{k-1}Y}\right]
\Big|\\
&\leq
p^{k-1}
\sum_{e\in \cT}
\E\left[
       |a_e-b_e|
       +b_e|G^e-G|
       + b_e|f(p^{k-1}Y, X+1)- f(p^{k-1} Y^e, X^e)|
\right]
\end{align*}
To complete the proof we use the properties~\eqref{eq:stein-bounds}.
\end{proof}

In the special case where $\cT\subseteq \cT(U)$ (as defined at~\eqref{eq:cTdef}) for $|U|\leq k$ we develop this bound further by bounding all terms on the RHS of the inequality in Lemma~\ref{lem:stein-app} except the change in $G$. 

We will use the following lemma in the proof of both of our remaining lemmas, Lemma~\ref{lem:fixedpointrewrite} and Lemma~\ref{lem:availability-target}.

\begin{lemma}\label{lem:stein-app2}
There exists a constant $C>0$ such that the following holds. Let $U\subseteq V$ with $|U|\leq k$ and let $\cT\subseteq \cT(U)$. Let $0\leq G=G(\mb I)\leq 1$. With $X,Y$ as in~\eqref{eq:XY-def} we have 
\[
\left|
\mathbb E\left[G\mathbf 1_{\{X=0\}}\right]
-
\mathbb E\left[Ge^{-p^{k-1}Y}\right]
\right|
\leq C\eta
+ 
p^{k-1}\sum_{e\in\cT}\E[b_e|G^e-G|]\, .
\]
\end{lemma}
\begin{proof}
We apply \cref{lem:stein-app} and look to bound the terms on the RHS of the inequality in the statement of that lemma.
First note that $b_e\leq a_e$ by definition. If $a_e(\mb I)=1$ but $b_e(\mb I)=0$, some edge $h\in E$ satisfies $|h\cap e|\ge2$ and $h\setminus e\subseteq \mb I$.
By \cref{obs:spatial-markov}, we therefore have the bound
\[
\mathbb E|a_e-b_e|
\le
\sum_{j=2}^{k-1}\binom {k-1}{j}\Delta_jp^{k-j}
\stackrel{\text{\cref{itm:setup Delta ell UB}}}{=}O(\eta)
\]
uniformly in $e$.

Next note that inserting $e$ can complete another member $f$ of $\cT$ only when $e\cap f\neq\emptyset$ and
$f\setminus e\subseteq \mb I$. More precisely, for $e\in \cT$ such that $e\not\subseteq \mb I$, we have
\[
0\leq b_e(\mb I)(X^e(\mb I) - X(\mb I) -1) \leq |\{f\in \cT\setminus \{e\} : f\cap e\neq \emptyset, f\setminus e \subseteq \mb I\}|\, .
\] 
On the other hand, if $e\subseteq \mb I$, then $X^e(\mb I)=X(\mb I)$.
By \cref{obs:spatial-markov} we conclude that 
\begin{align*}
\E\left[b_e |X^e-X-1| \right]
&\leq 
 p^{k-1}+ \sum_{\substack{f\in\cT\setminus \{e\}\\\ f\cap e\ne\emptyset}}
 p^{k-1-|e\cap f|}\\
 &\leq
 p^{k-1} + |U| \sum_{j=1}^{k-2} \binom{k-1}{j} \Delta_{j+1} p^{k-1-j}
 \stackrel{\text{\cref{itm:setup Delta ell UB}}}{\leq}
p^{k-1}+ O(\eta)=O(\eta)\, .
\end{align*}
Indeed, after a $j$-set $R\subset e$ and $u\in U$ has been chosen, there are at most $\Delta_{j+1}$ choices for $f\in E_{u}$ containing $R$. For the final inequality we used that $p=O(\eta)$ which follows from \cref{itm:setup Delta ell UB}, \cref{itm:setup p is critical} if $k\geq 3$
and
\cref{itm:setup codegree bound}, \cref{itm:setup p is critical} if $k=2$.

Next, for a vertex $u$, inserting $e$ can remove $f\in E_u$ from $Y_u$ only
if some $w\in f$ is newly blocked.  In that case there is an
$h\in E_w$ with $h\cap e\ne \emptyset$ and $h\setminus e\subseteq \mb I$. In other words, 
\[
0
\leq 
b_e(\mb I)(y_u(\mb I) - y^e_u(\mb I))
\leq
|\{f\in E_u: \text {there exists } w\in f, h\in E_w \text{ s.t. } h\cap e\ne \emptyset, h\setminus e\subseteq \mb I\}|.
\]

If $k\geq 3$, it follows from \cref{obs:spatial-markov} that 
\begin{align}\label{eq:ye-y}
\E\left [b_ep^{k-1}|y_u-y^e_u| \right]
\leq k\Delta p^{k-1}
\sum_{j=1}^{k-2}\binom{k-1}{j}\Delta_{j+1}p^{k-1-j} + \Delta_{k-1} p^{k-1} \Delta_2
\stackrel{\text{\cref{itm:setup Delta ell UB,itm:setup p is critical}}}{\leq}
O(\eta)\, .
\end{align}
Indeed, after choosing a vertex $w$ in the neighborhood of $u$ and a $j$-set $R\subseteq e$ there are $\Delta_{j+1}$ choices for the edge $h \in E_w$. To deal with the case where $h=e$, we note that there are at most $\Delta_{k-1}$ vertices $w$ that complete an edge with $e$, and that each is contained in at most $\Delta_2$ edges in $E_u$.

If $k=2$, write $e=\{w\}$ and note that inserting $w$ can only remove a vertex from $Y_u$ if that vertex is a common neighbor of $u$ and $w$. Thus,
\begin{align*}
\E\left [b_ep^{k-1}|y_u-y^e_u| \right]
\leq 
p^{k-1}(\codeg_{\mc H}(u,w)\mathbf 1_{\{u\neq w\}} + \Delta \mathbf 1_{\{u= w\}})
\leq O(\eta)\cdot \mathbf 1_{\{u\neq w\}} + O(1)\cdot \mathbf 1_{\{u= w\}}
\, .
\end{align*}
Together with~\eqref{eq:ye-y}, this gives, for all $k\geq2$,
\begin{align}\label{eq:y-to-ye-general}
\E[b_ep^{k-1}|y_u-y_u^e|]
=
O\left(\eta+\mathbf 1_{\{k=2,\ e=\{u\}\}}\right).
\end{align}
Since $\cT\subseteq\cT(U)$ and inserting $e$ can only remove
unblocked vertices, we have
\[
0\leq b_e(Y-Y^e)
\leq \sum_{u\in U}b_e(y_u-y_u^e).
\]
It follows that
\begin{align}
p^{2(k-1)}
\sum_{e\in\cT}\E[b_e|Y^e-Y|]
&\leq 
p^{2(k-1)}\sum_{e\in \cT}
\E[ b_e\sum_{u\in U}|y_u^e-y_u|]\nonumber
\\
&
=
O\left(
\eta |U|p^{k-1}|\cT|
+
\mathbf 1_{\{k=2\}}p|U|
\right)
=
O(\eta)\, ,\label{eq:sum-yue-bound}
\end{align}
where for the final inequality we used that $p^{k-1}|\cT|\leq |U|\Delta p^{k-1}=O(1)$ and $p=O(\eta)$. The result follows. 
\end{proof}

\subsection{The availability estimate}
We now turn our attention to proving Lemma~\ref{lem:availability-target}. Before proceeding, we derive one quick corollary of Lemma~\ref{lem:stein-app} needed for the proof. 

\begin{corollary}\label{cor:stein-app}
There exists a constant $C>0$ such that the following holds. Let $v\in V$ and let $\cT\subseteq E_v$. Let $0\leq G\leq 1$ be $\cW$-measurable. With $X,Y$ as in~\eqref{eq:XY-def} we have 
\[
\left|
\mathbb E\left[G\mathbf 1_{\{X=0\}}\right]
-
\mathbb E\left[Ge^{-p^{k-1}Y}\right]
\right|
\leq C(\eta +q)\, .
\]    
\end{corollary}
\begin{proof}
Condition on the window $W$ and apply Lemma~\ref{lem:stein-app2} with $U=\{v\}$. Since $G$ is $\cW$-measurable, $G^e(\mb I)$ and $G(\mb I)$ can only differ if $e\cap W\neq \emptyset$ i.e.
\[
b_e|G^e-G|\leq \mathbf 1_{\{e\cap W\neq \emptyset\}}\, .
\]
Taking expectations with respect to $W$, noting that $\P(e\cap W\neq \emptyset)\leq qk$ by a union bound, therefore gives 
\[
\left|
\mathbb E\left[G\mathbf 1_{\{X=0\}}\right]
-
\mathbb E\left[Ge^{-p^{k-1}Y}\right]
\right|
\leq C\eta + p^{k-1}|\cT|qk\, .
\]  
The result follows by noting that $kp^{k-1}|\cT|\leq kp^{k-1}\Delta=O(1)$.
\end{proof}

In order to set the stage for the proof of Lemma~\ref{lem:availability-target}, we recall some key definitions. Recall that $\vec{\alpha} = (\alpha_v)_{v\in V} \in [0,1]^V$ where the $\alpha_v$ are i.i.d.\ uniformly random values in $[0,1]$.

For $t\in [0,1]$, $\mathbb I_t=\{u\in \mathbb I: \alpha_u<t\}$,  
$
U_v(t)=\{v\in \mc U(\mathbb I_t) \},
$
and
$
u_v(t)
=
\P(U_v(t) \mid \cW)\, .
$
Recall also that $\mathbf y_v=\mathbf y_v(\mb I)=\sum_{e\in E_v} \mathbf 1_{\{e\subseteq \mc U(\mb I)\}}$.

Now let
\[
V_t=\{u\in V: \alpha_u<t\},\qquad
\cT_t=\{e\in E_v : e \subseteq  V_t\}\, .
\]
Define
\[
X_t=\sum_{e\in\cT_t}\mathbf 1_{\{e\subseteq \mb I\}},\qquad
Y_t=\sum_{e\in\cT_t} \mathbf 1_{\{e\subseteq \mc U(\mb I)\}}\ ,
\]
and note that $U_v(t)$ holds if and only if $X_t=0$. 
To prove Lemma~\ref{lem:availability-target}, we will apply Corollary~\ref{cor:stein-app} with these random variables. It will therefore be useful to first establish a comparison between $Y_t$ and $\mathbf y_v$.

\begin{lemma}\label{lem:Yt-to-bfyv}
We have 
\[
p^{k-1}\E|Y_t- t^{k-1}\mathbf y_v|=o(1)
\]
uniformly in $v\in V$ and $t\in [0,1]$.
\end{lemma}
\begin{proof}
First observe that
\[
\E_{\vec \alpha}[Y_t \mid \mb I] = t^{k-1} \mathbf y_v\, .
\]
Moreover, since the events $e\subseteq V_t$ and $f\subseteq V_t$ are independent unless $e\cap f\neq \emptyset$, we have
\[
\var_{\vec\alpha}(Y_t\mid\mb I) 
\leq
\sum_{e,f \in E_v} \mathbf 1_{\{e\cap f\neq \emptyset\}}
\leq
k\Delta\Delta_2
\leq 
k\eta\Delta^2
\]
where for the penultimate inequality we note that there are at most $\Delta$ choices for $e$, then $k-1$ choices for $u\in e$ and $\Delta_2$ choices for $f\in E_v$ containing $u$.

By Cauchy-Schwarz it follows that
\[
p^{k-1} \E|Y_t- t^{k-1}\mathbf y_v|
\leq 
p^{k-1} \E_{\mb I}\left[\sqrt{\var_{\vec\alpha}(Y_t\mid\mb I)}\right]
=
O(\sqrt{\eta})\, .
\]
\end{proof}

Finally we turn to the proof of Lemma~\ref{lem:availability-target}.
\begin{proof}[Proof of Lemma~\ref{lem:availability-target}]
Let $0\leq G\leq 1$ be $\cW$-measurable. 
Conditioning on $\vec \alpha$ (thus fixing $\cT_t$) and applying Corollary~\ref{cor:stein-app} we have
\[
\left|
\mathbb E\left[G\mathbf 1_{U_v(t)}\mid \vec\alpha\right]
-
\mathbb E\left[Ge^{-p^{k-1}Y_t} \mid \vec\alpha\right]
\right|
\leq C(\eta +q)\, .
\] 
Taking expectations with respect to $\vec\alpha$ then shows that
\begin{align}\label{eq:cor-stein-app}
\left|
\mathbb E\left[G\mathbf 1_{U_v(t)}\right]
-
\mathbb E\left[Ge^{-p^{k-1}Y_t} \right]
\right|
\leq C(\eta +q)
\end{align}
Since $|e^{-x}-e^{-y}|\leq |x-y|$ for $x,y\geq 0$ we have
\[
\E|Ge^{-p^{k-1}Y_t}- Ge^{-p^{k-1} t^{k-1}\mathbf y_v}|
\leq 
p^{k-1} \E|Y_t- t^{k-1}\mathbf y_v|
=o(1)
\]
where for the final equality we used Lemma~\ref{lem:Yt-to-bfyv}.
Note also that 
\begin{align}\label{eq:y-to-haty}
\E|Ge^{-p^{k-1} t^{k-1}\mathbf y_v}- Ge^{-p^{k-1} t^{k-1}\hat{\mathbf y}_v}|
\leq 
p^{k-1} \E|\mathbf y_v- \hat{\mathbf y}_v|=o(1)
\end{align}
where for the final equality we used Lemma~\ref{lem:y-estimator}.

Returning to~\eqref{eq:cor-stein-app} we then have
\begin{align}\label{eq:D-prec}
\left|
\mathbb E\left[G\mathbf 1_{U_v(t)}\right]
-
\mathbb E\left[Ge^{-p^{k-1} t^{k-1}\hat{\mathbf y}_v} \right]
\right|
=o(1)
\end{align}
where the error is uniform in $G$.
Now, since $G$ is $\cW$-measurable we have
\[
\mathbb E\left[G\mathbf 1_{U_v(t)}\right]
= 
\mathbb E\left[G\,\E[\mathbf 1_{U_v(t)}\mid\cW]\right] 
= 
\mathbb E\left[Gu_v(t)\right]\, .
\]
Letting $D_v(t)= u_v(t)-e^{-p^{k-1} t^{k-1}\hat{\mathbf y}_v}$~\eqref{eq:D-prec} can then be written as $|\E [GD_v(t)]|=o(1)$. Taking $G=\mathbf 1_{\{D_v(t)>0\}}$ and $G=\mathbf 1_{\{D_v(t)<0\}}$ (both $\cW$-measurable), we conclude that
\[
\E|D_v(t)| 
=
\E\left|
u_v(t)-e^{-p^{k-1} t^{k-1}\hat{\mathbf y_v}}
\right|=o(1)\, .
\]
The result now follows from~\eqref{eq:y-to-haty} and the triangle inequality. 
\end{proof}

To complete the proof of Lemma~\ref{lem:Bethe FE upper bound} it remains only to prove Lemma~\ref{lem:fixedpointrewrite}.

\subsection{Approximate BP fixed points}

We begin by showing that given a small set of vertices $U$, the probability that all vertices $u\in U$ are unblocked can be approximated by an averaged product of the
local quantities $\exp(-p^{k-1}\mathbf y_u)$.
 The proof is another application of Poisson approximation given by Lemma~\ref{lem:stein-app2}.

\begin{lemma}\label{lem:availability product decomposition}
Let $\mb I \sim \mu_{\mc H,p}$, and let $U,S \subseteq V$ be disjoint vertex sets with $U \in \mc I( \mc H)$ and $|U|+|S| \leq k$. There holds, uniformly in $U,S$,
\[
\E \left[ \mathbf{1}_{U \subseteq \mc U(\mb I)} \exp \left( -p^{k-1}\sum_{s \in S} \mathbf y_s \right) \right]
=
\E\left[\exp\left(-p^{k-1} \sum_{u\in U \cup S}\mathbf{y}_u\right)\right] + o(1).
\]
\end{lemma}
\begin{proof}
Let
\[
\cT:=\bigcup_{u\in U}E_u,\qquad
G:=\exp\left(-p^{k-1}\sum_{s\in S}y_s\right).
\]
Define $X,Y$ as in~\eqref{eq:XY-def} and observe that
 $U\subseteq \mc U(\mb I)$ if and only if $X=0$. We apply Lemma~\ref{lem:stein-app2} and observe that 
 \begin{align}\label{eq:Y-to-y}
0\le
p^{k-1}\sum_{u\in U}y_u-p^{k-1}Y
\le
p^{k-1}\sum_{\substack{u,u'\in U\\u\ne u'}}|E_u\cap E_{u'}|
\stackrel{\text{\cref{itm:setup codegree bound}}}{\leq} 
|U|^2 \eta p^{k-1}\Delta \stackrel{\text{\cref{itm:setup p is critical}}}{=} O(\eta)\, .
\end{align}
To complete the proof we note that 
\[
p^{k-1}\sum_{e\in \cT}
\E [b_e|G^e-G|] \leq p^{2(k-1)}\sum_{e\in \cT}
\E[ b_e\sum_{s\in S}|y_s^e-y_s|] =  O(\eta),
\]
where for the final equality we sum~\eqref{eq:y-to-ye-general} over $e\in \cT, s\in S$ and bound the resulting expression as in~\cref{eq:sum-yue-bound}.
\end{proof}

\begin{remark}\label{remark:empirical}
\cref{lem:availability product decomposition} shows in particular that $\P(v\in \mc U(\bb I))=\E[e^{-p^{k-1}\mathbf y_v}]+o(1)$ for all $v\in V$. Combined with the fact that $\P(v\in \bb I)= p \P(v\in \mc U(\bb I))$ we see that $\P(v\in \bb I)=\E[\mathbf x_v]+o(p)$ thus motivating the  name `empirical marginal' for $\mathbf x_v$.
\end{remark}

We are ready to prove \cref{lem:fixedpointrewrite} thereby completing the proof of \cref{lem:Bethe FE upper bound}.

\begin{proof}[Proof of \cref{lem:fixedpointrewrite}]
By the definitions of $\mc B$ and $\widetilde{\mc B}$,
\begin{align}
\widetilde{\mc B}_{\mc H,p}(\overline{\mathbf x})
-\mc B_{\mc H,p}(\overline{\mathbf x})
&=
\sum_{e\in E}\prod_{u\in e}\mathbf x_u
-\frac1k\sum_{v\in V}\mathbf x_v\log(p/\mathbf x_v)\nonumber\\
&=
\sum_{e\in E}\prod_{u\in e}\mathbf x_u
-\frac1k\sum_{v\in V}\mathbf x_vp^{k-1} \mathbf y_v \label{eq:final-B-Diff}
\end{align}
where for the final equality we recalled that $\mathbf x_v = pe^{-p^{k-1} \mathbf y_v}$. Recalling also that $\mathbf y_v=\sum_{e\in E_v}\mathbf 1_{\{e\subseteq \mc U(\mb I)\}}$ we have
\begin{align*}
p^{k-1}\E[\mathbf x_v \mathbf y_v]
=
p^k\sum_{e\in E_v} 
\E\left[
e^{-p^{k-1} \mathbf y_v}\mathbf 1_{\{e\subseteq \mc U(\mb I)\}}
\right]
&=
p^k\sum_{e\in E_v}\E\left[\exp\left(-p^{k-1} \sum_{u\in \{v\} \cup e}\mathbf{y}_u\right)\right] + o(p^k \Delta)\\
&=\sum_{e\in E_v}\E\left[\prod_{u\in \{v\} \cup e} \mathbf x_u\right]
+
o(p)
\end{align*}
where for the penultimate inequality we used Lemma~\ref{lem:availability product decomposition} with $U=e$ and $S=\{v\}$.
Summing over $v\in V$ yields
\[
p^{k-1}\sum_{v\in V}
\E[\mathbf x_v \mathbf y_v]
=
k
\sum_{e\in E}\E\left[\prod_{u\in  e} \mathbf x_u\right]
+
o(p n)\, .
\]
The result follows by returning to \eqref{eq:final-B-Diff} and taking expectations.
\end{proof}

\section{The lower bound in \cref{thm:Bethe exactness}}\label{sec:lower bound}

The main goal of this section is to prove the following.

\begin{lemma}\label{lem:Bethe FE lower bound}
    In the setup of \cref{thm:Bethe exactness}, with $\lambda=p/(1-p)$, it holds that
    \begin{equation}
    \log Z_{\mc H}(\lam) \geq \max_{x\in [0,1]^V} \mathcal{B}_{\mathcal{H},p}(x) + np - o(np).
    \end{equation}
\end{lemma}

We note that Theorem~\ref{thm:Bethe exactness} follows immediately from Lemmas~\ref{lem:Bethe FE lower bound},~\ref{lem:Bethe FE upper bound} and Observation~\ref{obs:hardcore partition and log probability}.
 
Fix a BP fixed point $x$ with respect to $\cH, p$ and 
recall the BP-guided random greedy process corresponding to $x$ from Section~\ref{subsec:BP-guided}. Recall also from~\eqref{eq:av-Def} that for $v\in V, t\in [0,1]$ we let $a_v(t)$ denote the probability that $v$ is unblocked when its clock rings at time $t$. With the heuristic analysis from Section~\ref{subsec:BP-guided} in mind, 
the main goal of this section is to prove the following lemma.

\begin{lemma}\label{lem:root-availability}
We have
\[
a_v(t)= (x_v/p)^{t^{k-1}} + o(1)\, .
\]
uniformly over all $v\in V$ and $t\in[0,1]$.
\end{lemma}

It will be useful to organize our computation of $a_v(t)$ via a \emph{computational tree}.
For this, we use the construction of a self-avoiding-walk tree for hypergraphs from~\cite{jenssen2026non}.

\subsection{The self-avoiding walk tree}
Fix $v\in V(\cH)$. A \emph{self-avoiding walk} (SAW) starting at $v$ is an
alternating sequence
\[
\mathsf u=(v_0,e_1,v_1,\ldots,e_d,v_d)
\]
such that $v_0=v$, the vertices $v_0,\ldots,v_d$ are distinct,
the hyperedges $e_1,\ldots,e_d$ are distinct, and
$\{v_{i-1},v_i\}\subseteq e_i$ for every $1\leq i\leq d$.
Its \emph{length} is $d$, and its \emph{label} is
$\ell(\mathsf u)=v_d$.

We can construct a linear hypertree $\bb{T}_v$ of self-avoiding walks in $\cH$ starting at $v$ (this tree is denoted by $T_{\mathrm{SAW},v}(\cH)$ in~\cite{jenssen2026non}).
   To avoid confusion we will refer to vertices of $\cH$ as \emph{vertices}, and vertices of $\bb{T}_v$ as \emph{nodes}. Furthermore, we use $u,v,w\ldots$ to denote vertices of $\cH$ and $\sf{u},\sf{v},\sf{w},\ldots$ to denote nodes of the tree. To define the construction formally, we first introduce some notation.

 Given a SAW $\sf{w}=(v_0,e_1,\ldots, v_{d})$, we say that a SAW $\sf{w}'$ \emph{extends} $\sf{w}$, denoted $\sf{w}\prec \sf{w}'$, if $\sf{w}'=(v_0,e_1,\ldots, v_{d}, e_{d+1}, v_{d+1})$. We now define $\bb T_v$ recursively, starting with the root
$\sf{r}=(v)$.
For each node $\sf{w}$ already constructed, attach a hyperedge
\begin{align}\label{eq:tree-hyperedge}
\{\sf{w},\sf{w}_1,\ldots,\sf{w}_{k-1}\}
\end{align}
for every collection of distinct SAWs
$\sf{w}_1,\ldots,\sf{w}_{k-1}$ satisfying:
\begin{enumerate}
    \item $\sf{w}\prec\sf{w}_i$ for every $i$;
    \item $\sf{w}_1,\ldots,\sf{w}_{k-1}$ differ only in their
    final vertex;
    \item
    $\{\ell(\sf{w}),\ell(\sf{w}_1),\ldots,
    \ell(\sf{w}_{k-1})\}\in E(\cH)$.
\end{enumerate}
We call $\sf{w}$ the \emph{parent node} of the hyperedge~\eqref{eq:tree-hyperedge}
and the parent of each node $\sf{w}_i$. 
For each $\sf{w}_i$, we refer to the hyperedge~\eqref{eq:tree-hyperedge} as the \emph{parent edge} of $\sf{w}_i$ and denote it by $h(\sf w_i)$.
The root has no parent node, and we set $h(\sf{r})=\emptyset$.
Conversely, we call the hyperedges~\eqref{eq:tree-hyperedge} the \emph{child edges} of $\sf w$ and we call the $\sf w_i$ the \emph{children} of $\sf w$. Given a subgraph $T\subseteq \bb T_v$, we say that $\sf u\in V(T)$ is a \emph{leaf} of $T$ if it has no children in $T$.

 Note that if $\sf{e}=\{\sf{w}_1, \ldots, \sf{w}_k\}$ is an edge of $\bb T_v$ then $\ell(\sf{e})=\{\ell(\sf{w}_1), \ldots, \ell(\sf{w}_k)\}$ is an edge of $\cH$ and we refer to $\ell(\sf{e})$ as the label of $\sf{e}$. 
 If for $\sf{w}, \sf{w}'$ there exists a sequence $\sf{w}=\sf{w}_1\prec \ldots \prec \sf{w}_\ell=\sf{w}'$ we say that $\sf{w}'$ is a \emph{descendant} of $\sf{w}$ and $\sf{w}$ is an \emph{ancestor} of $\sf{w}'$. We refer to the tree induced by $\sf{w}$ and its descendants as the \emph{subtree rooted at $\sf{w}$}. If the SAW $\sf{w}$ has length $\ell$ we say $\sf{w}$ is at \emph{depth} $\ell$ in $\bb T_v$. The depth of a tree hyperedge is one plus the depth of its
parent node.
 We note that by construction $\bb T_v$ is a $k$-uniform, linear hypertree of maximum degree at most $\Delta$.

\subsection{The shared clock process on $\mathbb{T}_v$.} We construct a random (not necessarily independent) set $I\subseteq V(\mathbb T_v)$. The process is similar to the BP guided process on $\cH$ (as defined in Section~\ref{subsec:BP-guided}) and we will argue that this new process gives an alternative (and useful) way to calculate $a_v(t)$. One small but key difference is our definition of being unblocked. Let 
\[
E'_\sf u
=
\{\sf f\setminus\{\sf u\}: \sf f \text{ is a child edge of } \sf u\}\, .
\]
We say that a node $\sf u\in V(\mathbb T_v)$ is \emph{available} with respect to a set $I\subseteq V(\mathbb T_v)$ if no $\sf e\in E'_{\sf u}$ is contained in $I$. Note that we do not require $I$ to be independent and only child edges incident to $\sf u$ determine its availability.

Assign to each node $\sf u$ the pair of random variables
$(\tau_{\ell(\sf u)}, \xi_{\ell(\sf u)})$
from~\eqref{eq:tau-xi}, and recall the definition of the functions $r_u$ from \eqref{eq:r_v definition}.

Start with $I=\emptyset$. If $\tau_{\ell(\sf u)}=t$ we add $\sf u$ to $I$ if both of the following hold:
\begin{enumerate}
\item no $\sf e\in E'_{\sf u}$ is already contained in $I$,
\item $\xi_{\ell(\sf u)}\leq r_{\sf \ell(\sf u)}(t)$.
\end{enumerate}

Let $a^{\textrm{sh}}_v(t)$ denote the probability that the root of $\mathbb{T}_v$ is available at time $t$ (note that here the clock $\tau_v$ cannot affect the availability status of $v$).
We begin with the observation that $a_v(t), a^{\textrm{sh}}_v(t)$ are identical. To prove this it will be useful to introduce a bit of terminology. 

Let us call any collection $(\tau_\sf u,\xi_\sf u)_{\sf u\in V(\mathbb T_v)}$ where  $\tau_\sf u, \xi_\sf u \sim U[0,1]$ for all $\sf u$ a \emph{clock process}. 

Consider a clock process $\omega$. Given $\sf u\in V(\mathbb T_v)$, we call $\sf e\in E'_{\sf u}$ \emph{influential at time} $t$ if  
\[
\tau_{\sf w}
< 
t
\quad\text{and}\quad
\xi_{\sf w}
\leq r_{\sf \ell(\sf w)}(\tau_{\sf w})
\quad\text{for each }\sf w\in \sf e
\, .
\]
Informally, $e$ is influential at time $t$ if it has the potential to block $\sf u$ at time $t$. 

For a node $\sf u$ and $t\in[0,1]$, we construct the
\emph{influence cluster} at time $t$, denoted $\mathcal J_{\sf u}(\omega,t)$,
recursively as follows. Initially include $\sf u$ and all nodes of every $\sf e\in E'_\sf u$
that is influential at time $t$. Whenever a node $\sf w\neq \sf u$
is included, add all nodes of every $\sf e\in E'_\sf w$
that is influential at time $\tau_{\sf w}$.
Continue until no further nodes can be added.
We note that $\mathcal{J}_{\sf u}(\omega, t)$ determines whether $\sf u$ is available at time $t$.

The following lemma is similar in spirit to the classical theorem of Weitz~\cite{weitz2006counting} (and its recent generalizations to hypergraphs~\cite{ bencs2025optimal,jenssen2026non}) which identifies the marginal of a vertex $v$ in the hardcore model on a graph, to the root marginal of the hardcore model on the self-avoiding walk tree $\bb T_v$ under certain boundary conditions.

\begin{lemma}\label{lem:Weitz}
For all $v\in V$ and $t\in [0,1]$
\[
a_v(t)=a^{\textrm{sh}}_v(t).
\]
\end{lemma}
\begin{proof}
Fix $v\in V$ and $t\in [0,1]$ and condition on $\tau_v=t$.
Recall that we couple the shared clock process on $\mathbb T_v$ and the random greedy process on $\cH$, by giving each node $\sf u \in V(\mathbb T_v)$ the same mark (clock and coin) as $\ell(\sf u)\in V(\cH)$.  Since $v$ labels only the root in $\bb T_v$, whose availability depends only on descendant marks, the conditioning does not change $a_v^{\textrm{sh}}(t)$. Let $\mathcal J_\sf r(t)$ denote the influence cluster at the root $\sf r$ at time $t$ under the shared clock process. 

Consider a non-root leaf $\sf w \in \mathcal J_\sf r(t)$ and let $\{\sf e_1, \ldots, \sf e_d\}$ denote its incident edges in $\bb T_v$ where $\sf e_1$ is the parent edge.
Note that $\{\ell(\sf e_1), \ldots, \ell(\sf e_d)\}$ is a subset of the edges incident to $\ell(\sf w)$ in $\cH$. Moreover, $\sf e_i$ and $\ell(\sf e_i)$ carry precisely the same set of marks.  

By the SAW tree construction, if $f$ is incident to $\ell(\sf w)$ and $f\notin \{\ell(\sf e_1), \ldots, \ell(\sf e_d)\}$ then $f$ contains the label of a strict ancestor of $\sf w$. Since $\sf w\in\mathcal J_\sf r(t)$, all such ancestors have clocks strictly greater than $\tau_{\sf w}$. Thus $f$ cannot block $\ell(\sf w)$ when its clock rings.

Since $\sf w$ is a leaf of $\mathcal J_\sf r(t)$, none of the
sets $\sf e_i':=\sf e_i\setminus\{\sf w\}$, $i\geq2$, is influential at time $\tau_{\sf w}$. It follows that none of
$\{\ell(\sf e'_2),\ldots,\ell(\sf e'_d)\}$ blocks $\ell(\sf w)$ when its clock rings. Moreover, the parent edge $\sf e_1$ carries a clock strictly greater than $\tau_{\sf w}$, so $\ell(\sf e'_1)$ cannot block $\ell(\sf w)$ at this time either. Thus, $\sf w$ and $\ell(\sf w)$ are both available when their common clock rings, and their shared coin gives the same acceptance decision.

 Continuing recursively from the leaves to the root it follows that $\sf r$ is available at time $t$ if and only if $\ell(\sf r)=v$ is available when its clock rings at time $t$.
\end{proof}

\subsection{The independent clock process on $\mathbb{T}_v$.} In order to understand the shared clock process, it will be useful to first study the following `uncoupled' version where  each node $\sf u$ is assigned $\tau_{\sf u}, \xi_{\sf u}\sim U[0,1]$ independently for every node $\sf u$. As before, we start with $I=\emptyset$, and if $\tau_{\sf u}=t$ we add $\sf u$ to $I$ if both of the following hold:
\begin{enumerate}
\item no $\sf e\in E'_{\sf u}$ is already contained in $I$,
\item $\xi_{\sf u}\leq r_{\sf \ell(\sf u)}(t)$.
\end{enumerate}
We let $a^{\textrm{ind}}_v(t)$ denote the probability that the root of $\mathbb{T}_v$ is available at time $t$ in this process. Due to the independence, $a^{\textrm{ind}}_v(t)$ is easier to analyse. We will therefore first estimate $a^{\textrm{ind}}_v(t)$  and then argue that $a^{\textrm{ind}}_v(t)$ is close to $ a^{\textrm{sh}}_v(t)$. 
Before estimating $a^{\textrm{ind}}_v(t)$, we record the following integral form of Gronwall's inequality that will come in handy (see e.g.~\cite[Lemma 2.7]{teschl2012ordinary}). 

\begin{lemma}\label{lem:Gronwall}
Let $T>0$ and let $f:[0,T]\to[0,\infty)$ be integrable. Suppose that for constants $a,b\geq 0$
\[
f(t) 
\leq
a+b\int_{0}^t f(s)\, ds
\]
for all $t\in [0,T]$. Then for all $t\in [0,T]$
\[
f(t)\leq ae^{bt}\, .
\]
\end{lemma}

We now prove the following estimate on $a^{\textrm{ind}}_v(t)$.

\begin{lemma}\label{lem:ind-availability}
We have
\[
a^{\textrm{ind}}_v(t)= (x_v/p)^{t^{k-1}} + o(1)\, ,
\]
uniformly over all $v\in V$ and $t\in[0,1]$.
\end{lemma}
\begin{proof}
Consider the independent clock process on $\mathbb{T}_v$ as described above. 
For $\sf u\in V(\mathbb T_v)$, let 
\[
\widetilde a_{\sf u}(t)
=
\P( \sf u \text { available at time } t),
\qquad
\widetilde o_{\sf u}(t)
=
\P( \sf u \text { occupied at time } t)\, .
\]
As suggested by our heuristic analysis in Section~\ref{subsec:heuristic}, we aim to show that $\widetilde o_{\sf u}(t)$ is approximately equal to $t x_{\ell(\sf u)}$. With this in mind, we define
\[
D(t)
:=
\sup_{\sf u}\sup_{0\leq s\leq t}
\frac{1}{d_{\sf u}+1}
\left|
\frac{\widetilde o_{\sf u}(s)}
{x_{\ell(\sf u)}}-s
\right|
\]
where $d_{\sf u}$ denotes the depth of the node $\sf u$.

Observe that by conditioning on the value of the clock $\tau_\sf u$, we obtain the relation
\begin{align}\label{eq:o-cond-tau}
\widetilde o_{\sf u}(t)
=
\int_0^t
\widetilde a_{\sf u}(s)
r_{\ell(\sf u)}(s)\, ds\, .
\end{align}
Since $r_{\ell(\sf u)}(s)\leq p$ we have the bound
\begin{align}\label{eq:tilde-o-bound}
\widetilde o_{\sf u}(t)
\leq 
pt\, .
\end{align}
Note also that due to the independent marks, we have
\begin{align}\label{eq:tree-recursion}
\widetilde a_{\sf u}(t)
=
\prod_{\sf e\in E'_{\sf u}}
\left( 
1
- 
\prod_{\sf w \in \sf e}
\widetilde o_{\sf w}(t)
\right)\, .
\end{align}
For $\sf e\in E'_{\sf u}$, let 
$
\widetilde o_\sf e(t)
=
\prod_{\sf w \in \sf e} 
\widetilde o_{\sf w}(t).
$
Taking logarithms in~\eqref{eq:tree-recursion} then gives
\begin{align}\label{eq:log-available}
\log \widetilde a_{\sf u}(t)
=
-\sum_{e\in E'_{\sf u}}
\left(
\widetilde o_{\sf e}(t) 
+ O(\widetilde o^2_{\sf e}(t))
\right)
=
-\sum_{e\in E'_{\sf u}}
\widetilde o_{\sf e}(t) 
+
O(p^{k-1})
\end{align}
where for the final equality we used~\eqref{eq:tilde-o-bound} and that $|E'_{\sf u}|\leq\Delta$ and $p^{k-1}\Delta=O(1)$.

Note that by the definition of $D(t)$, the fact that every child of $\sf u$ has depth $d_\sf u+1$, and the basic inequality 
$
|\prod_{i=1}^{k-1}z_i
- \prod_{i=1}^{k-1}w_i|
\leq
\theta^{k-2}
\sum_{i=1}^{k-1}
|z_i-w_i|
$
for all $w_i,z_i\in \R$ such that $|z_i|, |w_i|\leq \theta$ for all $i$
we have
\begin{align}\label{eq:o-product-diff}
\left|
\widetilde o_{\sf e}(t)
-
t^{k-1}
\prod_{\sf w \in \sf e}
x_{\ell(\sf w)}
\right|
\leq
(pt)^{k-2}(d_\sf u+2)D(t)
\sum_{\sf w \in \sf e}
x_{\ell(\sf w)}\, .
\end{align}
Let 
$
L_{\sf u}
:=
\sum_{e\in E_{\ell(\sf u)}}
\prod_{w \in  e}
x_{w}
$
and suppose that $\sf u=(v_0,e_1,v_1,\ldots,e_d,v_d)$ (so in particular $d=d_\sf u$).
Note that $e\in E_{\ell(\sf u)}$ is the label of some $\sf e \in E'_\sf u$ unless it contains some $v_i$ where $i<d$. Each such $v_i$ belongs to at most $\Delta_2$ members of $E_{\ell(\sf u)}$ and so at most $d_{\sf u}\Delta_2$ members of $E_{\ell(\sf u)}$ are not represented in $E'_{\sf u}$.
Therefore
\[
0\leq
L_{\sf u}
-
\sum_{e\in E'_{\sf u}}\prod_{\sf w\in e}x_{\ell(\sf w)}
\leq d_{\sf u}\Delta_2p^{k-1}.
\]
Summing~\eqref{eq:o-product-diff} over
$\sf e\in E'_{\sf u}$ and returning to
\eqref{eq:log-available} gives
\begin{align*}
\left|
\log\widetilde a_{\sf u}(t)
+t^{k-1}L_{\sf u}
\right|
&\leq
(pt)^{k-2}(d_{\sf u}+2)D(t)
\sum_{\sf e\in E'_{\sf u}}
\sum_{\sf w\in\sf e}x_{\ell(\sf w)}
+
O\bigl((1+d_{\sf u}\Delta_2)p^{k-1}\bigr)
\\
&\leq
C(d_{\sf u}+1)D(t)
+
O\bigl((1+d_{\sf u}\Delta_2)p^{k-1}\bigr),
\end{align*}
where the final inequality uses $\Delta p^{k-1}=O(1)$ and
\[
\sum_{\sf e\in E'_{\sf u}}
\sum_{\sf w\in\sf e}x_{\ell(\sf w)}
\leq (k-1)\Delta p\, .
\]

Noting that $\Delta_2p^{k-1}=o(1)$
we then have
\begin{align}\label{eq:a-final}
\left|
\widetilde a_{\sf u}(t)
-
\exp
\left(
-t^{k-1}L_{\sf u}
\right)
\right|
=
O
\left(
(d_\sf u+1)D(t)
\right) 
+
o(d_\sf u+1)\, 
\end{align}
where the implicit constant and function in the $O, o$ notation are uniform in $\sf u, v, t$.
Multiplying this inequality by $r_{\ell(\sf u)}(t)\leq p$, integrating, and recalling the identity~\eqref{eq:o-cond-tau} yields
\begin{align}\label{eq:o-final}
\left|
\widetilde o_{\sf u}(t)
-
\int_{0}^t
r_{\ell(\sf u)}(s)
\exp
\left(
-s^{k-1}L_{\sf u}
\right)
 ds
\right|
=
\left|
\widetilde o_{\sf u}(t)
-
t x_{\ell(\sf u)}
\right|
=
O
\left(
p
(d_\sf u+1)
\int_{0}^tD(s)\, ds
\right) +o(p(d_\sf u+1))\, .
\end{align}
where for the second inequality we
used that 
$
r_{\ell(\sf u)}(s)
\exp
(-s^{k-1}L_{\sf u})
=
x_{\ell(\sf u)}
$
by the definition of $r_{\ell(\sf u)}(s)$ and the fact that $x$ is a BP fixed point. Dividing by $(d_\sf u+1)x_{\ell(\sf u)}$, recalling that $x_{\ell(\sf u)}=\Theta(p)$ uniformly in $\sf u$, and taking suprema gives
\[
D(t) \leq C
\int_{0}^tD(s)\, ds
+o(1)
\]
for some constant $C$ and error term both uniform in $t$ and $v$.
Applying Gronwall's inequality (Lemma~\ref{lem:Gronwall}) we conclude that 
\[
D(t)=o(1)\, ,
\]
uniformly in $t, v$.
Returning to~\eqref{eq:a-final} at the root $\sf r$ where $d_\sf r=0$ and $L_\sf r=\log (p/x_v)$ completes the proof. 
\end{proof}

It remains to show that $a^{\textrm{ind}}_v(t)$ is close to $a^{\textrm{sh}}_v(t)$. 
\begin{lemma}\label{lem:as-close}
We have
\[
|a^{\textrm{ind}}_v(t)-a^{\textrm{sh}}_v(t)|=o(1)
\]
uniformly in $v\in V$ and $t\in [0,1]$.
\end{lemma}

\subsection{Intermediate clock processes on $\bb T_v$}
To prove this lemma, it will be useful to `interpolate' between the shared and independent clock processes. We do this by going through the nodes that are coupled to another node one by one, and introducing a clock and coin that is independent from everything else. This leads us to \emph{intermediate clock processes} which we define as follows. First recall that in the independent clock process, each node $\sf u\in V(\bb T_v)$ is assigned an independent mark $(\tau_{\sf u}, \xi_{\sf u})$ which is uniform in $[0,1]^2$. Now, given a map $\rho: V(\bb T_v)\to V(\bb T_v)$ such that $\ell(\rho(\sf u))=\ell(\sf u)$, we let $\omega_\rho$ denote the \emph{intermediate} clock process where we assign the node $\sf u$ the mark $(\tau_{\rho(\sf u)}, \xi_{\rho(\sf u)})$. In other words, two nodes use the same mark precisely when they have the same image under $\rho$, otherwise their marks are independent. In particular, nodes with distinct labels receive independent marks. Note that $\omega_\rho$ defines a random set $I\subseteq V(\mathbb T_v)$ in the same way as in the independent/shared clock processes.

If $\rho$ is the identity map, then $\omega_\rho$ is the independent clock process. If for each $x\in V(\cH)$, we pick a unique representative $\sf w_x\in V(\bb T_v)$ such that $\ell(\sf w_x)=x$ and define $\rho(\sf u)= \sf w_{\ell(\sf u)}$ for all $\sf u$, then $\omega_\rho$ is the shared clock process.  Henceforth, let $\omega_{\textrm{ind}}, \omega_{\textrm{sh}}$ denote the independent and shared clock process respectively. 

Roughly speaking, we will show that if availability of the root differs under $\omega_{\textrm{ind}}, \omega_{\textrm{sh}}$, then we can find two distinct nodes of $\mathbb T_v$ with the same label that influenced the root in some intermediate model. We then show this event is unlikely, using the fact that $\cH$ is $\eta$-locally tree-like with $\eta\to 0$.

We now make this formal. 

Let $\omega$ be a clock process. For $m\in\mathbb N$, it will also be useful to consider the \emph{truncated} influence cluster $\mathcal{J}^m_\sf u(\omega,t)$ where we stop building the influence cluster at $\sf u$ under $\omega$ at time $t$ (following some arbitrary fixed ordering in which to process nodes and edges) as soon as its number of nodes is $> m$. Thus $|\mathcal{J}^m_\sf u(\omega,t)|\leq m+k-1$.

Let $A(\omega,t)$ denote the indicator of the event that the root $\sf r$ is available under $\omega$ at time $t$. Let $A_m(\omega,t)=A(\omega,t) \mathbf 1_{\{|\mathcal{J}_\sf r(\omega,t)|\leq m\}}$.  Lemma~\ref{lem:as-close} will follow easily from the following two lemmas. 

\begin{lemma}\label{lem:fourcorners}
Fix $m\in \N$, and let $F(\omega,t)=A_m(\omega,t)$ or $F(\omega,t)=\mathbf 1_{\{|\mathcal{J}_\sf r(\omega,t)|\leq m\}}$. Then 
\[
\left|
\E(F(\omega_\textrm{sh},t))
-
\E(F(\omega_\textrm{ind},t))
\right|
=o(1)\, ,
\]
uniformly in $t\in [0,1]$.
\end{lemma}

\begin{lemma}\label{lem:inf1stmoment}
Uniformly over $v\in V$ and $t\in [0,1]$,
\[
\E(|\mathcal J_{\sf r}(\omega_{\textrm{ind}},t)|)=O(1)
\]
\end{lemma}

Before proving these lemmas we first show how they imply Lemma~\ref{lem:as-close}.

\begin{proof}[Proof of Lemma~\ref{lem:as-close}]
Fix an $m\in \N$ and recall that $A_m(\omega,t)=A(\omega,t) (1-\mathbf 1_{\{|\mathcal{J}_\sf r(\omega,t)|> m\}})$ and so
\begin{align*}
|a^{\textrm{ind}}_v(t)-a^{\textrm{sh}}_v(t)|
&=
|\E (A(\omega_\textrm{ind},t)) -
\E (A(\omega_\textrm{sh},t))|\\
&\leq
|\E (A_m(\omega_\textrm{ind},t)) -
\E (A_m(\omega_\textrm{sh},t))|
+\P(|\mathcal{J}_\sf r(\omega_\textrm{ind},t)|> m)
+\P(|\mathcal{J}_\sf r(\omega_\textrm{sh},t)|> m)\, .
\end{align*}
By Lemma~\ref{lem:fourcorners} with 
$F(\omega)=A_m(\omega,t)$ we have
$|\E (A_m(\omega_\textrm{ind},t)) -
\E (A_m(\omega_\textrm{sh},t))|=o(1)$. By Lemma~\ref{lem:inf1stmoment} and Markov's inequality, $\P(|\mathcal{J}_\sf r(\omega_\textrm{ind},t)|> m)=O(1/m)$. Applying Lemma~\ref{lem:fourcorners}, this time with $F(\omega)=\mathbf 1_{\{|\mathcal{J}_\sf r(\omega,t)|\leq m\}}$, then gives $\P(|\mathcal{J}_\sf r(\omega_\textrm{sh},t)|> m)=O(1/m)+o(1)$. The result follows by taking $m\to\infty$ sufficiently slowly. 
\end{proof}

It remains to prove Lemmas~\ref{lem:fourcorners} and \ref{lem:inf1stmoment}. We begin with Lemma~\ref{lem:inf1stmoment}.

\begin{proof}[Proof of Lemma~\ref{lem:inf1stmoment}]
Let $N_d$ denote the number of nodes in $\mathcal J_{\sf r}(\omega_{\textrm{ind}},t)$ at depth $d$ and note that $N_d\leq ((k-1)\Delta)^d$. Suppose $\sf w\in \mathcal J_{\sf r}(\omega_{\textrm{ind}},t)$ has depth $d$ and let $P_\sf w$ denote the path from $\sf r$ to $\sf w$.  If $\sf u_1, \ldots, \sf u_{d-1}$ are the nodes of degree $>1$ in $P_\sf w$, we must have $\tau_\sf w<\tau_{\sf u_{d-1}}<\ldots<\tau_{\sf u_1}$ which occurs with probability $1/d!$. Moreover, all the non-root nodes of $P_\sf w$ must have passed their coin tests which, conditioned on any outcome of the clocks at these node, occurs with probability $\leq p^{(k-1)d}$. We conclude that 
\[
\E N_d \leq \frac{((k-1)\Delta p^{k-1})^d}{d!}\, .
\]
Thus
\[
\E|\mathcal J_{\sf r}(\omega_{\textrm{ind}},t)|
\leq
\sum_{d\geq0}
\frac{((k-1)\Delta p^{k-1})^d}{d!}
=
\exp((k-1)\Delta p^{k-1})
=
O(1).
\]
\end{proof}

 We now turn to Lemma~\ref{lem:fourcorners}. We begin with a basic observation.
 
 \begin{observation}\label{obs:F-diff}
 Fix $m\in \N$ and let $F$ be as in Lemma~\ref{lem:fourcorners}.
 Suppose that $\omega, \omega'$ are two clock processes on $\bb T_v$ that differ only in the marks at nodes $\sf U\subseteq V(\bb T_v)$. Suppose that $F(\omega,t)\neq F(\omega',t)$,  then  $\sf U\cap (\mc J^m_\sf r(\omega,t) \cup \mc J^m_\sf r(\omega',t))\neq \emptyset$.
 \end{observation}
\begin{proof}
First note that $F(\omega,t)$ is determined by the influence cluster $\mc J^m_\sf r(\omega,t)$ and so $F(\omega,t)\neq F(\omega',t)$ implies that $\mc J^m_\sf r(\omega,t)\neq \mc J^m_\sf r(\omega',t)$. Let $\sf e$ be the first set that is included in $\mc J^m_\sf r(\omega,t)$ but not $\mc J^m_\sf r(\omega',t)$ (or vice versa) in the exploration process that defines the truncated influence cluster. Let $\sf x$ denote the parent node of $\sf e$. Since $\sf x\in \mc J^m_\sf r(\omega,t)$ we may assume that $\sf x \notin \sf U$ (else we are done). Let $\tau_\sf x$ denote the clock at $\sf x$ under $\omega, \omega'$ (where we set $\tau_\sf x=t$ if $\sf x$ is the root). It then follows that $\sf e$ is influential at time $\tau_\sf x$ under $\omega$ but not under $\omega'$ (or vice versa). Since $\omega, \omega'$ are identical except for the marks assigned to $\sf U$ it must then be the case that $\sf U\cap \sf e\neq \emptyset$. 
\end{proof}

Before proving Lemma~\ref{lem:fourcorners} we also need the following lemma which shows that in any intermediate clock model, it is unlikely to find two distinct nodes of the same label in the influence cluster. In fact, we need the following slightly stronger statement.

\begin{lemma}\label{lem:influence-same-label}
Fix $m$ and for each pair $\sf u, \sf w\in V(\mathbb T_v)$ fix a pair of intermediate processes $\omega_{\sf u, \sf w}, 
\omega'_{\sf u, \sf w}$.
We have, uniformly in $t\in [0,1]$,
\[
\sum_
{\substack
{
\sf u \neq \sf w
\\
\ell(\sf u) = \ell(\sf w)
}}
\P
\left(
\sf u 
\in
\mathcal J_\sf r^m(\omega_{{\sf u, \sf w}},t)
,
\sf w 
\in
\mathcal J_\sf r^m(\omega'_{{\sf u, \sf w}},t)
\right)
=
o(1)\, .
\]
\end{lemma}
\begin{proof}
Fix $t\in[0,1]$ and let $\mathcal J_\sf r^m(\omega)=\mathcal J_\sf r^m(\omega,t)$. All of the bounds below are independent of $t$.
Let $M:=m+k-1$ and recall that $|\mathcal J_\sf r^m(\omega)|\leq M$ for all $\omega$. Given $\sf u\in V(\mathbb T_v)$, let $P_\sf u$ denote the the unique path from $\sf r$ to $\sf u$ in $\mathbb T_v$.
Observe that if $\sf u\in \mathcal J_\sf r^m(\omega)$ for some $\omega$ then $P_\sf u\subseteq \mathcal J_\sf r^m(\omega)$ and so $|V(P_\sf u)|\leq M$.

Given $S\subseteq V(\mathbb T_v)$, let 
\[
\|S\|=|\ell(S \setminus \{\sf r\})|\, .
\]

Note that for the event 
$
\{
\sf u 
\in
\mathcal J_\sf r^m(\omega_{{\sf u, \sf w}})
,
\sf w 
\in
\mathcal J_\sf r^m(\omega'_{{\sf u, \sf w}})
\}
$
to occur, a coin at each non-root node in $P_\sf u \cup P_\sf w$ must have passed its test. Since coins with distinct labels are independent, it follows that 
\[
\P
\left(
\sf u 
\in
\mathcal J_\sf r^m(\omega_{{\sf u, \sf w}})
,
\sf w 
\in
\mathcal J_\sf r^m(\omega'_{{\sf u, \sf w}})
\right)
\leq 
p^{\|P_\sf u \cup P_\sf w\|}\, .
\]
Fix $a$ and let 
\[
N_a
:=
\left|
\left\{
(\sf u, \sf w): 
\sf u \neq \sf w,\,  
\ell(\sf u)= \ell(\sf w),\, 
|V(P_\sf u)|, |V(P_\sf w)|\leq M,\, 
\|P_\sf u \cup P_\sf w\|=a
\right\}
\right|
\, .
\]
Since there are $\leq 2M$ choices for $a$ and $p\Delta^{1/(k-1)}=O(1)$, it will suffice to show that $N_a=o(\Delta^{a/(k-1)})$ to conclude the proof.
To help with the counting, we make some additional specifications for the structure of $P_\sf u, P_\sf w$. First, suppose that $P_\sf u, P_\sf w$ have length $t,s$ respectively and that they share a prefix of length $r$, that is, $P_\sf u=(\sf e_1,\ldots, \sf e_t)$ and $P_\sf w=(\sf e_1,\ldots, \sf e_r, \sf e_{t+1},\ldots, \sf e_{t+s-r})$. Furthermore, let us specify the equivalence relation that  specifies which nodes in $P_\sf u, P_\sf w$ share the same label (without specifying the label itself). Since $|V(P_\sf u)|, |V(P_\sf w)|\leq M$ there are $O_M(1)$ choices for these specifications. Let 
\[
m_i=
\left|
\ell(\sf e'_i) \setminus \bigcup_{j<i} \ell(\sf e_j)
\right|
\]
where $\sf e_i'$ denotes $\sf e_i$ with its parent node removed.
Note that the sequence $(m_i)$ is determined by the above specifications. Moreover, we have $0\leq  m_i\leq k-1$ for all $i$ and $\sum m_i=a$.

We choose the edges of $P_\sf u \cup P_\sf w$ one by one in the order $\sf e_1, \ldots, \sf e_{t+s-r}$. Suppose that we have chosen $\sf e_1, \ldots, \sf e_{i-1}$ and that $m_i=q$. This means that when we choose $\sf e_i$ we reveal $q$ new labels and $k-1-q$ labels in $\sf e_i$ have already been revealed. There are $O_M(1)$ ways to choose these $k-1-q$ labels and to choose the parent node for $\sf e_i$. Once these have been specified, there are at most $\Delta_{k-q}$ choices for $\sf e_i$. It follows that there are $O_M(\Delta_{k-q})$ choices for this step. Recall that $\Delta_{k-q}=O(\Delta^{q/(k-1)})$ and if $q\in \{1,\ldots, k-2\}$ we have the stronger bound $\Delta_{k-q}=o(\Delta^{q/(k-1)})$. Therefore if one of the $m_i$ lies in $\{1,\ldots, k-2\}$, the number of pairs $\sf u, \sf w$ producing the sequence $(m_i)$ is at most
\[
o\left(
\Delta^{\sum m_i/(k-1)}
\right)
=
o\left(
\Delta^{a/(k-1)}
\right)\, .
\]
It remains to consider the case where $m_i\in  \{0,k-1\}$ for all $i$. 

Let $j$ be the minimal index for which $m_j=0$ and note that $j\geq 1$ exists since $\ell(\sf u)=\ell(\sf w)$. 
Let $\cT$ denote the subtree of $\mathbb T_v$ consisting of the edges $\sf e_1 \cup\ldots\cup \sf e_{j-1}$.
Since $m_i=k-1$ for all $i<j$, there is a unique set 
$
\sf U \subseteq  V(\cT)
$
such that $\ell(\sf U)=\ell(\sf e_j)$. Let $\cS\subseteq \cT$ be the minimal connected subtree of $\cT$ containing $\sf U$ and the root $\sf r$. Let $\sf e$ be the edge of maximum depth in $\cS$ and let $\sf x$ be its parent node. Let $\cS'$ denote $\cS$ with all the nodes in $\sf e\setminus \{\sf x\}$ removed and note that $\cS'$ is connected and contains the root $\sf r$. 

We will count the number of choices for $P_\sf u, P_\sf w$ as follows. 
\begin{enumerate}
\item Choose $\cS'$ edge by edge in such a way that an edge incident to the root is chosen first and we only choose the next edge once its parent has been chosen.
\item Choose the pair $(\sf e, \sf e_j)$.
\item Choose the remaining edges, edge by edge, in such a way that we only choose an edge once its parent has been chosen.
\end{enumerate}
Let $n_1,\ldots, n_c$ be the number of new labels revealed at each edge revelation in Step 1. Let $n_{c+1},\ldots, n_d$ be the number of new labels revealed at each edge revelation in Step 3. Note that by the definitions of $\cS, \cT$ and the fact that $m_j=0$, we must reveal exactly $k-1$ new labels when we reveal $(\sf e, \sf e_j)$. It follows that $\sum_i n_i=a-k+1$. Arguing as above, we have 
\[
O_M\left(\Delta^{\sum n_i/(k-1)}\right)
=
O_M\left(\Delta^{a/(k-1)-1}\right)
\]
 choices for the edges in Steps 1 and 3. Thus, to show that $N_a=o(\Delta^{a/(k-1)})$, it suffices to show that the number of choices for $(\sf e, \sf e_j)$ in Step 2 is $o(\Delta)$.

To this end, let $\sf e' = \sf e \setminus \{ \sf x \}$ and let $h=\ell(\sf e_j)\cap \ell(\sf e')$ so that $1\leq |h| \leq k-1$. Note that in Step 1 we chose $\ell(\sf e_j) \setminus h$ and so there are at most $\Delta_{k-|h|}$ choices for $\ell(\sf e_j)$. Once  $\ell(\sf e_j)$ has been chosen, its new vertices (those not chosen in Step 1) form the set $h$. We note that $\ell(\sf e)$ contains $h$ along with the label of the parent node of $\sf e$ which was chosen at Step 1. There are therefore at most $\Delta_{|h|+1}$ choices for $\ell(\sf e)$. The joint number of choices for $(\sf e, \sf e_j)$ is therefore at most
\[
O_M(\Delta_{k-|h|} \Delta_{|h|+1})
\]
which is $o(\Delta)$ if $1\leq |h| \leq k-2$. 

It remains to consider the case where $|h|=k-1$. First we note that $\ell(\sf e)\neq \ell(\sf e_j)$. Indeed, if $\sf e, \sf e_j$ are incident, then by construction of $\bb T_v$, either $\ell(\sf e), \ell(\sf e_j)$ are consecutive edges in a self-avoiding walk in $\cH$ or they are two distinct extensions of a self-avoiding walk in $\cH$. If $\sf e, \sf e_j$ are not incident, then the label of the parent node of $\sf e_j$ must not appear in $\ell(\sf e)$ since $\sf e\in\{\sf e_1, \ldots, \sf e_{j-1}\}$, the parent node of $\sf e_j$ is contained in $\sf e_1\cup  \ldots \cup \sf e_{j-1}$ and all the labels of the nodes $\sf e_1\cup  \ldots \cup \sf e_{j-1}$ are distinct (since $m_i=k-1$ for all $i<j$). Now, let $y$ be the unique element of $\ell(\sf e_j)\setminus h$. We have $\ell(\sf e_j)=h\cup\{y\}$ and $\ell(\sf e)=h\cup \{x\}$ where $x$ is the label of the parent node of $\sf e$. Since $\ell(\sf e)\neq \ell(\sf e_j)$ we have $x\neq y$, and so there are at most $o(\Delta)$ choices for $h$. We conclude that there are at most $\codeg_{\mc H}(x,y)=o(\Delta)$ choices for $(\sf e, \sf e_j)$ as desired. 
\end{proof}

\begin{proof}[Proof of Lemma~\ref{lem:fourcorners}]
 Recall that in the independent clock process, each node $\sf u\in V(\bb T_v)$ is assigned an independent mark $(\tau_{\sf u}, \xi_{\sf u})$ which is uniform in $[0,1]^2$.
Fix $u\in V(\cH)$ and let 
$\sf u_1,\ldots, \sf u_a$
denote all the nodes in $\mathbb T_v$ with label $u$. 
Under $\omega_\textrm{sh}$, each $\sf u_i$ receives the same mark 
$
\Gamma:=(\tau_{\sf u_a}, \xi_{\sf u_a})
$
say.
Our goal is to `uncouple' these marks by assigning each $\sf u_i$ its own independent mark one by one. 
To this end, let 
$\Gamma_i=(\tau_{\sf u_i}, \xi_{\sf u_i})$ for all $i\in\{1,\ldots,a-1\}$.
 For $i\in\{0,\ldots, a-1\}$, let $\omega_i$ denote the (intermediate) process obtained from $\omega_\textrm{sh}$ where $\sf u_1,\ldots, \sf u_{i}$ are given new marks $\Gamma_1,\ldots, \Gamma_{i}$ respectively. In particular, $\omega_0=\omega_\textrm{sh}$. 
Fix $t\in [0,1]$ and let $F(\omega):=F(\omega,t)$ and $\mathcal{J}^m_\sf r(\omega)=\mathcal{J}^m_\sf r(\omega,t)$. 
Our goal is to compare $F(\omega_0)$ and $F(\omega_{a-1})$. We do this in steps, by writing
\begin{align}\label{eq:telescope1}
F(\omega_0)-F(\omega_{a-1})=\sum_{i=1}^{a-1} F(\omega_{i-1})-F(\omega_{i})\, .
\end{align}
Fix $i\in \{1,\ldots, a-1\}$ and let $\omega_{i-1}'$ denote the process obtained from $\omega_{i-1}$ where $\sf u_{i},\dots, \sf u_{a}$ are given the new mark $\Gamma_{i}$. Let $\omega_{i}'$ denote the intermediate process obtained from $\omega_{i-1}$ where $\sf u_{i+1},\dots, \sf u_{a}$ are given the new mark $\Gamma_{i}$. We note that by symmetry (under swapping $\Gamma,\Gamma_{i}$),
\begin{align}\label{eq:Gamma-swap}
\E
\left[
F(\omega_{i-1})
-
F(\omega_{i})
\right]
=
\frac{1}{2}
\E
\left[
F(\omega_{i-1})
-
F(\omega_{i})
+
F(\omega'_{i-1})
-
F(\omega'_{i})
\right]\, .
\end{align}
Let $D_{i}$ denote the quantity inside the expectation on the RHS. Note that if $D_{i}\neq 0$, then either $F(\omega_{i-1})\neq F(\omega_{i})$ or $F(\omega'_{i-1})\neq F(\omega'_{i})$. In the former case, note that $\omega_{i-1}, \omega_{i}$ are identical aside from the mark at $\sf u_i$
and so $\sf u_i$ is contained in $\mathcal{J}^m_\sf r(\omega_{i-1})$ or $\mathcal{J}^m_\sf r(\omega_{i})$ by Observation~\ref{obs:F-diff}.
 Similarly if $F(\omega'_{i-1})\neq F(\omega'_{i})$ then $\sf u_i$ is contained in $\mathcal{J}^m_\sf r(\omega'_{i-1})$ or $\mathcal{J}^m_\sf r(\omega'_{i})$. To summarise, if $D_i\neq 0$ then $\sf u_i$ must lie in 
\[
\mathcal J(i):=  
\mathcal{J}^m_\sf r(\omega_{i-1})\cup \mathcal{J}^m_\sf r(\omega_{i})\cup \mathcal{J}^m_\sf r(\omega'_{i-1})\cup \mathcal{J}^m_\sf r(\omega'_{i})\, .
\]

On the other hand, if $D_{i}\neq 0$, then we can also conclude that either $F(\omega_{i-1})\neq F(\omega'_{i})$ or $F(\omega_{i})\neq F(\omega'_{i-1})$. Note that $\omega_{i-1}, \omega'_{i}$ are identical aside from the marks at $\sf u_{i+1},\ldots, \sf u_a$ and similarly for $\omega_{i}, \omega'_{i-1}$. Again by Observation~\ref{obs:F-diff}, it follows that one of $\sf u_{i+1},\ldots, \sf u_a$ must lie in $\mathcal J(i)$.

Since $|D_{i}|\leq 2$, we conclude that
\[
\E |D_{i}|
\leq 
2
\P
\left(
\sf u_i \in \mathcal J(i),
\{\sf u_{i+1},\ldots, \sf u_a\}
\cap
\mathcal J(i)
\neq\emptyset
\right)
\leq
2
\sum_{j=i+1}^a
\P
\left(
\sf u_i, \sf u_j \in \mathcal J(i)
\right)\, .
\]
Returning to~\eqref{eq:telescope1},
~\eqref{eq:Gamma-swap}, we conclude that
\[
\left|\E
\left[
F(\omega_{0})
-
F(\omega_{a-1})
\right]
\right|
\leq
\sum_{i=1}^{a-1}
\sum^a_{j=i+1}
\P
\left(
\sf u_i, \sf u_j \in \mathcal J(i)
\right)\, .
\]
Taking $\omega_{a-1}$, uncoupling any other collection of nodes with the same label as above, repeating until we reach $\omega_{\textrm{ind}}$, and finally applying Lemma~\ref{lem:influence-same-label} completes the proof. 
\end{proof}

Lemma~\ref{lem:root-availability} now follows immediately from Lemmas~\ref{lem:Weitz}~\ref{lem:ind-availability} and~\ref{lem:as-close}.

\subsection{Bounding the entropy}
We are now ready to prove Lemma~\ref{lem:Bethe FE lower bound}.

\begin{proof}[Proof of Lemma~\ref{lem:Bethe FE lower bound}]
Let $x$ be a BP fixed point, and let $I$ denote the independent set arising from the guided random greedy process on $\cH$ corresponding to $x$. For $v\in V$, let $X_v=\mathbf 1_{\{v\in I\}}$.

Condition on the clock values $(\tau_v)_{v\in V}$ and expose the $X_v$ in increasing clock order. When $v$ is reached its state, blocked or unblocked, is determined by the preceding decisions. If $v$ is blocked then $X_v$ has entropy $0$. If $v$ is unblocked then $X_v$ is a Bernoulli random variable with mean $r_v(\tau_v)$. The entropy chain rule therefore gives the identity
\[
H(I \mid (\tau_v)_{v\in V})
=
\sum_{v\in V} \int_{0}^1 
a_v(t)
h(r_v(t))\, dt\, .
\]
Similarly,
\[
\E |I|
=
\sum_{v\in V} \int_{0}^1 
a_v(t)
r_v(t)\, dt\, .
\]
Since conditioning cannot increase entropy we conclude that
\[
H(I)
+
\log \lambda \E |I|
\geq 
\sum_{v\in V} \int_{0}^1 
a_v(t)
(
h(r_v(t))
+\log \lam \cdot 
r_v(t))
\, dt\, .
\]
Recall that $r_v(t)= p (x_v/p)^{1-t^{k-1}}=\Theta(p)$ and note that
\[
h(r_v(t))
+\log \lam \cdot 
r_v(t)
=
r_v(t)
\left( 
\log (\lambda/r_v(t))
+1
\right)
+O(p^2)
=
O(p)\, .
\]
It follows from Lemma~\ref{lem:root-availability} that 
\begin{align*}
H(I)
+
\log \lambda \E |I|
&\geq
\sum_{v\in V} \int_{0}^1 
(x_v/p)^{t^{k-1}}
(
h(r_v(t))
+\log \lam \cdot 
r_v(t))
\, dt
+o(np)\\
&=
\sum_{v\in V} \int_{0}^1 
(x_v/p)^{t^{k-1}}
r_v(t)
\left( 
\log (\lambda/r_v(t))
+1
\right)
\, dt
+o(np)\\
&=
\widetilde{\mathcal{B}}_{\mc H,p} (x) +np +o(np)\, .
\end{align*}
By Observation~\ref{obs:Gibbs} and
Observation~\ref{obs:Bethe equivalence at BPFP}, it follows that
\[
\log Z_{\mc H}(\lambda)
\geq
\mc B_{\mc H,p}(x)+np-o(np).
\]
Finally, Observation~\ref{obs:Bethe free energy properties}
allows us to choose $x$ to be a maximizer of
$\mc B_{\mc H,p}$, completing the proof.
\end{proof}

\section{A structural result for $\mu_{\cH,p}$}\label{sec:hypergraph-structure}
In this section, we prove a structural result for typical samples from $\mu_{\cH,p}$ that is a key step toward proving the structural portion of~\cref{thm:graphon clique avoidance}. 

Throughout this section we assume the setup of \cref{thm:Bethe exactness} and let $\bb I\sim \mu_{\cH,p}$. Recall also the definition of the empirical marginals $\mathbf x_v=\mathbf x_v(\bb I)$ from~\eqref{eq:empirical-marginals}. Recall also from \cref{sec:entropy-window} that the window $W$ is a $q$-random binomial subset of $V$ with $q=q(n)\to0$ and we write $\cW=(W,\bb I\cap W)$.

Our main focus in this section is to prove the following result which shows that the structure of $\bb I$ is typically captured by the empirical marginals $\mathbf x_v$. 

\begin{theorem}\label{thm:hypergraph-structure}
    Let $\cA\subseteq 2^V$ be non-empty such that $\log (2|\cA|)=o(np)$.
    \[
    \E\max_{A\in \cA}
    \left|
    |\bb I\cap A|-\sum_{v\in A}
    \mathbf x_v
    \right|
    =o(np).
    \]
\end{theorem}

As a look ahead, in the next section we will apply this result in the case where $\cH$ is the hypergraph encoding copies of a fixed graph $H$ in $K_n$. There we take $\cA$ to be the set of all cuts in $K_n$, thereby recovering the cut norm.

We next observe that, typically, the vector of empirical marginals $\ol{\bf x}=(\mathbf x_v)_{v\in V}$ almost maximizes the Bethe functional $\mathcal{B}_{\mc H,p}$. Indeed, comparing Lemmas~\ref{lem:Bethe FE lower bound} and~\ref{lem:partition function UB} and using Lemma~\ref{lem:fixedpointrewrite} shows that
\begin{align}\label{eq:mean-vs-max}
\max_{x\in [0,1]^V} \mathcal{B}_{\mc H,p}(x)
-
\E[\mathcal{B}_{\mc H,p}(\ol{\mathbf{x}})] 
=o(np)\, .
\end{align}
We therefore have the following immediate corollary.

\begin{corollary}\label{cor:bar x almost maximizes B}
Let $\cA\subseteq 2^V$ be non-empty such that $\log (2|\cA|)=o(np)$. Then whp both of the following hold:
\[
\mathcal{B}_{\mc H,p}(\ol{\mathbf{x}})
\geq \max_{x\in [0,1]^V} \mathcal{B}_{\mc H,p}(x)-o(np)
\]
and
\[
\left|
    |\bb I\cap A|-\sum_{v\in A}
    \mathbf x_v
    \right|
    =o(np)
    \quad \text{for all}\quad 
    A\in \cA\, .
\]
\end{corollary}

\begin{proof}
This follows from~\cref{eq:mean-vs-max} and \cref{thm:hypergraph-structure} and Markov's inequality. 
\end{proof}

The starting point for the proof of \cref{thm:hypergraph-structure} is the following simple corollary of \cref{lem:availability-target}. 
\begin{corollary}\label{cor:estimator-xv}
\[
    \E\sum_{v\in V}
    \left|
    \P(v\in \bb I \mid \cW) 
    -
    \mathbf x_v
    \right|
    =o(np).
\]
\end{corollary}
\begin{proof}
We recall from~\eqref{eq:occ-cond-window} that for $v\in V$,
$
\mathbf 1_{\{v\notin W\}}\P(v\in\mathbb I\mid\cW)=\mathbf 1_{\{v\notin W\}} p u_v(1).
$
On the other hand, by \cref{lem:availability-target} applied with $t=1$ we have
\[
    \E
    \left|
    pu_v(1) 
    -
    \mathbf x_v
    \right|
    =o(p)
\]
uniformly in $v\in V$. Summing over $v\notin W$ we conclude that
\[
    \E
    \sum_{v\in V}\mathbf 1_{\{v\notin W\}}\left|
    pu_v(1) 
    -
    \mathbf x_v
    \right|
    =
    \E
    \sum_{v\in V}\mathbf 1_{\{v\notin W\}}\left|
    \P(v\in \bb I\mid \cW) 
    -
    \mathbf x_v
    \right|
    =o(np)\, .
\]
Moreover, by the triangle inequality
\[
 \E
    \sum_{v\in V}\mathbf 1_{\{v\in W\}}\left|
    \P(v\in \bb I\mid \cW) 
    -
    \mathbf x_v
    \right|
    \leq \E |\bb I\cap W| + p\E|W|
    \leq 2npq
    =o(np)\, ,
\]
where for the first inequality we used that $ \mathbf x_v\leq p$ and for the second we recalled that $\P(v\in \bb I)\leq p$, $\P(v\in W)=q=o(1)$ and $W,\bb I$ are independent.
\end{proof}

Given $A\subseteq V$, let $N_A:=|\bb I \cap A|$ and let $\hat N_A$ denote the estimator $\E[N_A\mid \cW]$. Observe that~\cref{cor:estimator-xv} allows us to control the distance between $\hat N_A$ and $\sum_{v\in A}\mathbf x_v$. 
\cref{thm:hypergraph-structure} will therefore follow easily from the following lemma which shows that $\hat N_A$ is a good estimator for $N_A$. 
\begin{lemma}\label{lem:NA-hatNA}
Let $\cA\subseteq 2^V$ be non-empty such that $\log (2|\cA|)=o(np)$. Then
\[
\E\max_{A\in \cA}
|
N_A - \hat N_A
|
=
o(np)\, .
\]
\end{lemma}
\begin{proof}
For $A\in \cA$, let $\tilde N_A:=q^{-1}|\bb I \cap A \cap W|$. First note that it suffices to show that 
$\E\max_{A\in \cA}
|
N_A - \tilde N_A
|
=
o(np)$ since we have
\begin{align*}
\E\max_{A\in \cA}
|
\hat N_A - \tilde N_A
|
=
\E\max_{A\in \cA}
\left|
\E[ N_A - \tilde N_A \mid \cW]
\right|
&\leq
\E\max_{A\in \cA}
\E[| N_A - \tilde N_A| \mid \cW]\\
&\leq 
\E
\E[\max_{A\in \cA}| N_A - \tilde N_A| \mid \cW]\\
&=\E\max_{A\in \cA}| N_A - \tilde N_A|\, .
\end{align*}
Now observe that conditioned on the event $\{\bb I=I\}$, $N_A=|I\cap A|$ and $q\tilde N_A\sim \text{Bin}(N_A, q)$. Thus, by Chernoff's inequality
\[
\P_W(|\tilde N_A - N_A|>t \mid \mathbb I=I)
\leq
2\exp\left( 
-\frac{q t^2}{2|I\cap A|+t}
\right)
\]
Noting that $|I\cap A|\leq |I|$ and taking a union bound over all $A\in \cA$ gives
\[
\P_W(\max_{A\in \cA}|\tilde N_A - N_A|>t \mid \mathbb I=I)
\leq
2|\cA|\exp\left( 
-\frac{q t^2}{2|I|+t}\,
\right).
\]
We now integrate this estimate over $t$ to bound the expectation. 
Fix $\eps\in(0,1)$ and note that if $t\geq \eps |I|$, then $q t^2/(2|I|+t)\geq \eps qt/3$. Thus for $T\geq \eps |I|$,
\[
\E_W(\max_{A\in \cA}|\tilde N_A - N_A| \mid \mathbb I=I)
\leq 
T 
+ 
\int_{T}^\infty 2|\cA|
\exp\left( 
-\frac{\eps q t}{3}
\right)\, dt 
=
T+\frac{6|\cA|}{\eps q}
\exp\left( 
-\frac{\eps q T}{3}
\right)\, .
\]
Setting $T=\eps |I| + 3/(\eps q)\log |\cA|$ yields
\[
\E_W(\max_{A\in \cA}|\tilde N_A - N_A| \mid \mathbb I=I)
\leq 
\eps |I| + \frac{3}{\eps q}\log |\cA| + \frac{6}{\eps q}\, .
\]
Taking expectations with respect to $I$ and noting that $\E |\bb I|\leq np$ by \cref{obs:spatial-markov} gives
\[
\E(\max_{A\in \cA}|\tilde N_A - N_A|)
\leq 
\eps np + \frac{3}{\eps q}\log |\cA| + \frac{6}{\eps q}\, .
\]
The result follows by taking $\eps, q\to 0$ sufficiently slowly. Such a choice is possible since the assumption $\log(2|\mc A|) = o(np)$ implies that $np = \omega(1)$.
\end{proof}

\cref{thm:hypergraph-structure} now follows quickly. 
\begin{proof}[Proof of \cref{thm:hypergraph-structure}]
 The triangle inequality and \cref{cor:estimator-xv} imply that 
\[
\E\max_{A\in \cA}
|
\hat N_A - \sum_{v\in A} \mathbf x_v
|
\leq
\E\sum_{v\in V}
    \left|
    \P(v\in \bb I \mid \cW) 
    -
    \mathbf x_v
    \right|
=
o(np)\, .
\]
The result now follows from \cref{lem:NA-hatNA} and the triangle inequality. 
\end{proof}

\section{Graphon formulation for subgraph avoidance}\label{sec:Graphon}
\newcommand{\AutH}{{\mr{Aut}(H)}}

A fundamental problem in the theory of large deviations is to determine the rate function for non-existence of a fixed graph $H$ in $G(n,p)$ and to find the structure of a typical $H$-free sample from $G(n,p)$. In this section we apply \cref{thm:Bethe exactness} to these problems in the case that $H$ is strictly $2$-balanced and $p$ is critical with respect to $H$.

Recall that for a graph $H$ we define
\[
m_2(H) \coloneqq \max_{\substack{F \subseteq H:\\|V(F)|\geq3}} \frac{e(F)-1}{v(F)-2}.
\]
A graph $H$ with $e(H) \geq 3$ is \textit{strictly $2$-balanced} if this maximum is uniquely achieved at $F=H$. Notably, all cliques on at least three vertices are strictly $2$-balanced. On the other hand, a clique with a pendant edge is not.

Recall the definitions of $\mc W$, $t(H,w)$ and $I(w)$ from \cref{sec:intro avoiding subgraphs}. Analogously to $\mc B_c^{K_r}$ we define, for a graphon $w$,
\[
\mc B_c^H(w) \coloneqq - \frac{1}{|\AutH|} c^{e(H)} t(H,w) - cI(w).
\]
We also define
\[
F_H(c) \coloneqq \max_{w \in \mc W} \mc B_c^H(w)
\]
(we prove the maximum exists below). Our main theorem on $H$ avoidance in $G(n,p)$, which implies \cref{thm:graphon clique avoidance}, is the following. Recall that for a graph $G$, we write $\tilde G$ for the step graphon it induces, and we write $\delta_\Box(\cdot,\cdot)$ for the cut distance. We extend the cut distance in the obvious way to bounded symmetric kernels (such as the normalized graphon $\tilde G/p$).

\begin{theorem}\label{thm:H non-existence}
    Let $H$ be a strictly $2$-balanced graph, let $c>0$, and let $p = cn^{-1/m_2(H)}$. Then
    \[
    \lim_{n \to \infty} n^{1/m_2(H)-2} \log \left( \Pr[H \nsubseteq G(n,p)] \right) = F_H(c).
    \]
    Moreover, if $G \sim G(n, p)$ conditioned on being $H$-free, then whp there exists a graphon $w$ that maximizes $\mc B_c^{H}$ and such that $\delta_\Box(\tilde G/p,w) = o(1)$.
\end{theorem}

We observe that \cref{thm:H non-existence} implies \cref{thm:graphon clique avoidance}, since $K_r$ is strictly $2$-balanced and $m_2(K_r) = (r+1)/2$.

In order to prove \cref{thm:H non-existence}, we first show that it fits the framework of \cref{thm:Bethe exactness}. To this end, for a graph $H$, we define the constraint hypergraph $\mc H = \mc H(H,n)$ as follows: the vertices are the $\binom{n}{2}$ edges of $K_n$ (viewed as the host graph of $G(n,p)$), and the hyperedges correspond to copies of $H$ in $K_n$. The following lemma follows straightforwardly from the definitions. It is proved in \cite[Lemma 7.1]{jenssen2026non}; we include the proof here for completeness.

\begin{lemma}\label{lem:Bethe exactness applies to strictly 2 balanced graphs}
    Let $H$ be a strictly $2$-balanced graph and let $c>0$. There exists a constant $\delta = \delta(H)$ such that $\mc H = \mc H(H,n)$ is $n^{-\delta}$-locally-tree-like, and $p=cn^{-1/m_2(H)}$ is $\Theta_c(1)$-critical with respect to $\mc H$.
\end{lemma}

\begin{proof}
    Let $\delta = \delta(H)$ be a sufficiently small constant. We write $r = v(H)$ and $k = e(H)$.

    We first note that every vertex in $\mc H$ has degree $\Delta(\mc H) = \frac{2k(r-2)!}{|\AutH|}\binom{n-2}{r-2}$.
    
    Next, we observe that if $e$ and $f$ are distinct edges in $K_n$, then they jointly span at least three vertices. Additionally, since $H$ is strictly $2$-balanced, it has minimum degree at least two. Thus, any edge set that completes each of them to a copy of $H$ must have at most $r-3$ vertices not in $e \cup f$. Therefore, $\max_{e,f} \codeg_{\mc H}(e,f) = O(n^{r-3}) = O(\Delta(\mc H)/n)$, and \cref{itm:setup codegree bound} is satisfied.

    That \cref{itm:setup Delta ell UB} holds follows from the definition of strict $2$-balance. Indeed, this definition implies that there exists some $\varepsilon>0$ such that
    \[
    r - v(F) < r - 2 - \frac{e(F)-1}{m_2(H)} - \varepsilon
    \]
    for every $F \subsetneq H$ with at least three vertices. Let $\ell \in \{2,\ldots,k-1\}$ and let $F' \subseteq E(K_n)$ have $\ell$ edges. Note that $v(F') \geq 3$. The number of copies of $H$ containing $F'$ is at most the number of ways to identify $F'$ with a subgraph $F$ of $H$, times the number of ways to extend this identification to a copy of $H$ in $K_n$. There are $O_H(1)$ ways to identify $F'$ with a subgraph $F \subseteq H$. It follows that the number of ways to complete $F'$ to a copy of $H$ is at most
    \begin{align*}
    O \left( n^{r - v(F')} \right) & = O \left( n^{r - 2 - (\ell-1)/m_2(H) - \varepsilon} \right) = O \left( \Delta(\mc H)^{1 - (\ell-1)/((r-2)m_2(H))} \times n^{- \varepsilon} \right)\\
    & = O \left( \Delta(\mc H)^{1 - (\ell-1)/(k - 1)} \times n^{- \varepsilon} \right) = O \left( \Delta(\mc H)^{(k - \ell)/(k - 1)} \times n^{- \varepsilon} \right),
    \end{align*}
    so \cref{itm:setup Delta ell UB} holds.

    Finally, we note that
    \begin{align*}
    p\Delta(\mc H)^{1/(k-1)} & = cn^{-1/m_2(H)} \times \left( \frac{2k(r-2)!}{|\AutH|}\binom{n-2}{r-2} \right)^{1/(k-1)}\\
    & = (1+o(1)) c \left( \frac{2k}{|\AutH|} \right)^{1/(k-1)}.
    \end{align*}
    It follows that $p$ is $\Theta_c(1)$-critical with respect to $\mc H$, as desired.
\end{proof}

We now prove that $F_H(c)$ is well-defined.

\begin{lemma}
    Let $H$ be a graph and let $c>0$. The function $\mc B_c^H$ is bounded above on $\mc W$ and attains its supremum.
\end{lemma}

\begin{proof}
    Observe that for every $w \in \mc W$, both $t(H,w)$ and $I(w)$ are nonnegative. In the case of $t(H,w)$ this is apparent from the definition. In the case of $I(w)$ this follows from the fact that $x\log x + 1 - x \geq 0$ for every $x \in [0,1]$. Thus $\mc B_c^H(w) \leq 0$.

    The fact that $\mc B_c^H$ attains its maximum follows from compactness of the graphon space (i.e., $\mc W$ modulo measure preserving transformations) with respect to the cut metric, the fact that $t(H,w)$ and $I(w)$ are invariant under measure preserving transformations, and the facts that $t(H,w)$ is continuous and $I(w)$ is lower semicontinuous with respect to the cut metric (cf.\ \cite[Lemma 2.1]{chatterjee2011large}). These imply that $\mc B_c^H$ induces an upper semicontinuous functional on the graphon space and therefore attains its supremum.
\end{proof}

We now prove a connection between the discrete Bethe free energy $\mc B_{\mc H,p}$ and its continuous counterpart $\mc B_c^H$. In what follows, given an edge-weighting $w:E(K_n)\to[0,1]$, we identify it with the step-graphon that gives weight $w(ij)$ to the square $((i-1)/n,i/n)\times ((j-1)/n,j/n)$ for $i \neq j$ and weight $0$ everywhere else.

\begin{claim}\label{clm:discrete to continuous Bethe comparison}
    Let $H$ be strictly $2$-balanced and let $c>0$. Let $p = c n^{-1/m_2(H)}$. Let $w:E(K_n)\to[0,p]$. Then
    \[
    n^{1/m_2(H)-2}\mc B_{\mc H(H,n),p}(w) = \mc B_c^H(w/p) + o(1)
    \]
    uniformly over $w$.
\end{claim}

\begin{proof}
    We write $r = v(H)$ and $k = e(H)$. We consider the two contributions to $\mc B_c^H(w/p)$ and $\mc B_{\mc H(H,n),p}(w)$ separately. First, by definition
    \begin{align*}
        t(H,w/p) & = \int_{[0,1]^{v(H)}} \prod_{ij \in E(H)} \frac{w(y_i,y_j)}{p} dy_1\cdots dy_{v(H)}\\
        & = \frac{1}{n^{r}} \sum_{\substack{x_1,\ldots,x_r \in [n]\\x_1,\ldots,x_r\text{ distinct}}} \prod_{ij\in E(H)} \frac{w(x_ix_j)}{p} + O \left( \frac{n^{r-1}}{n^r} \right)\\
        & = \frac{|\AutH|}{n^r p^k} \sum_{\substack{H' \subseteq K_{n}\\H' \cong H}} \prod_{e \in E(H')}w(e) + O \left( \frac{n^{r-1}}{n^r} \right)\\
        & = n^{1/m_2(H) - 2} \frac{|\AutH|}{c^k} \sum_{\substack{H' \subseteq K_{n}\\H' \cong H}} \prod_{e \in E(H')}w(e) +  O \left( \frac{n^{r-1}}{n^r} \right).
    \end{align*}
    Therefore
    \[
    \frac{1}{|\AutH|} c^k t(H,w/p) = n^{1/m_2(H) - 2} \sum_{\substack{H' \subseteq K_{n}\\H' \cong H}} \prod_{e \in E(H')}w(e) - o(1).
    \]

    We now consider $I(w/p)$. By definition
    \begin{align*}
        c I(w/p) & = \frac{c}{2} \iint_{x,y \in [0,1]} \left( \frac{w(x,y)}{p} \log \left(\frac{w(x,y)}{p} \right) + 1 - \frac{w(x,y)}{p} \right)dxdy\\
        & = \frac{c}{2n^2} \sum_{\substack{i,j \in [n]\\i \neq j}} \left( \frac{w(ij)}{p} \log \left( \frac{w(ij)}{p}\right) + 1 - \frac{w(ij)}{p} \right) + O(n^{-1})\\
        & = n^{1/m_2(H)-2} \sum_{ij \in E(K_{n})} \left( w(ij) \log\left(\frac{w(ij)}{p}\right) + p - w(ij) \right) + o(1).
    \end{align*}
    It follows that
    \begin{align*}
    \mc B_c^H(w/p) & = - \frac{1}{|\AutH|} c^k t(H,w/p) - cI(w/p)\\
    & = n^{1/m_2(H)-2} \left( - \sum_{\substack{H' \subseteq K_{n}\\H' \cong H}} \prod_{ij \in E(H')}w(ij) - \sum_{ij \in E(K_{n})} \left( w(ij) \log\left(\frac{w(ij)}{p}\right) + p - w(ij) \right) \right) - o(1)\\
    & =n^{1/m_2(H)-2} \mc B_{\mc H(H,n),p}(w) + o(1),
    \end{align*}
    as claimed.
\end{proof}

We now prove the first assertion in \cref{thm:H non-existence}.
\begin{lemma}\label{lem:H rate function}
    Let $H$ be a strictly $2$-balanced graph and let $c>0$. Then
    \[
    \lim_{n \to \infty} n^{1/m_2(H)-2} \log \left( \Pr[H \nsubseteq G(n,cn^{-1/m_2(H)})] \right) = F_H(c).
    \]
\end{lemma}
\begin{proof}
    For conciseness, define $f(n) \coloneqq \log \left( \Pr[H \nsubseteq G(n,cn^{-1/m_2(H)})] \right)$. We will prove the two inequalities
    \begin{equation}\label{eq:FHC limsup}
        \limsup_{n \to \infty} n^{1/m_2(H)-2} f(n) \leq F_H(c)
    \end{equation}
    and
    \begin{equation}\label{eq:FHC liminf}
        F_H(c) \leq \liminf_{n \to \infty} n^{1/m_2(H)-2} f(n).
    \end{equation}
    Together these imply the first assertion of \cref{thm:H non-existence}.

    We begin with \eqref{eq:FHC limsup}. Let $\{n_\ell\}_{\ell\in\N}$ be a sequence such that
    \[
    \lim_{\ell \to \infty} n_\ell^{1/m_2(H)-2} f(n_\ell) = \limsup_{n \to \infty} n^{1/m_2(H)-2} f(n).
    \]
    Let $p_\ell \coloneqq c n_\ell^{-1/m_2(H)}$. By \cref{lem:Bethe exactness applies to strictly 2 balanced graphs,thm:Bethe exactness}, if $x_\ell$ is a maximizer of $\mc B_{\mc H(H,n_\ell),p_\ell}$, then
    \[
    f(n_\ell) = \mc B_{\mc H(H,n_\ell),p_\ell}(x_\ell) + o(n_\ell^2p_\ell).
    \]
    Additionally, recall from \cref{obs:Bethe free energy properties} that $x_\ell(e) \leq p_\ell$ for every edge $e \in E(K_{n_\ell})$. Noting that by definition $n_\ell^2p_\ell = c n_\ell^{2-1/m_2(H)}$, we see that
    \[
    \lim_{\ell\to\infty} n_\ell^{1/m_2(H)-2} \mc B_{\mc H(H,n_\ell),p_\ell}(x_\ell) = \limsup_{n \to \infty} n^{1/m_2(H)-2} f(n).
    \]

    For each $\ell$ we define the step graphon $w_\ell$ that is induced by $x_\ell/p_\ell$ in the obvious way. That is, we identify $V(K_{n_\ell})$ with the integer set $[n_\ell]$, and for each $i\neq j \in [n_\ell]$ we define $w_\ell$ to have constant weight $x_\ell(ij)/p_\ell$ on the square $((i-1)/n_\ell,i/n_\ell)\times((j-1)/n_\ell,j/n_\ell)$. On all remaining points of the square (including the diagonal squares of the form $((i-1)/n_\ell,i/n_\ell)^2$) we let $w_\ell$ have constant weight $0$. We observe that since $x_\ell \leq p_\ell$, the function $w_\ell$ is bounded between $0$ and $1$.

    We claim that
    \[
    F_H(c) \geq \mc B_c^H(w_\ell) = n_\ell^{1/m_2(H)-2} \mc B_{\mc H(H,n_\ell),p_\ell}(x_\ell) + o(1),
    \]
    which will imply \eqref{eq:FHC limsup}. Indeed, the inequality follows from the definition of $F_H(c)$. The equality follows from \cref{clm:discrete to continuous Bethe comparison}.

    We now prove \eqref{eq:FHC liminf}. Let $w$ be a graphon such that $\mc B_c^H(w) = F_H(c)$. For each $n \in \N$, let $p_n \coloneqq c n^{-1/m_2(H)}$. We will build a sequence of weightings $x_n : E(K_n) \to [0,1]$ such that $\lim_{n \to\infty} n^{1/m_2(H)-2} \mc B_{\mc H(H,n),p_n} (x_n) = \mc B_c^H(w)$, implying \eqref{eq:FHC liminf}. Concretely, for each $n$, define the step graphon $w_n$ by giving it weight $n^2 \iint_{((i-1)/n,i/n) \times ((j-1)/n,j/n)} w(x,y) dxdy$ on each square $((i-1)/n,i/n) \times ((j-1)/n,j/n))$ with $i\neq j$, and weight $0$ elsewhere. Note that $\{w_n\}$ converges to $w$ in $L^1$. We then define the weighting $x_n$ by
    \[
    x_n(ij) \coloneqq p_n w_n((i-1/2)/n,(j-1/2)/n).
    \]
    
    Since $w_n$ converges to $w$ in $L^1$ it follows that $t(H,w_n) \to t(H,w)$ and that $I(w_n) \to I(w)$. Therefore $\lim_{n\to\infty} \mc B_c^H(w_n) = \mc B_c^H(w) = F_H(c)$. Additionally, by \cref{clm:discrete to continuous Bethe comparison}, $n^{1/m_2(H)-2}\mc B_{\mc H(H,n),p_n}(p_nw_n) = \mc B_c^H(w_n)+o(1)$, implying \eqref{eq:FHC liminf} and completing the proof.
\end{proof}

Next, we prove that typical $H$-free graphs are close in cut metric to the empirical marginals $\mathbf{x}$. To recap the definitions, specialized to the case of $H$-free graphs, let $G \subseteq K_n$ be an $H$-free graph. Write $\mc U(G)$ for the set of edges $e \subseteq K_n$ such that $G+e$ is $H$-free. For each $e \in K_n$, write $y_e(G)$ for the number of copies $H' \subseteq K_n$ of $H$ containing $e$ such that $H' \setminus e \subseteq \mc U(G)$. Finally, define $\mathbf{x} = \mathbf{x}(G)$ by $\mathbf{x}_e = p \exp\left(-p^{e(H)-1} y_e\right)$. Additionally, we prove that $\mathbf{x}$ is usually close to a maximizer of $\mc B_c^H$.

\begin{lemma}\label{lem:H structure}
    Let $H$ be a strictly $2$-balanced graph and let $c>0$. Let $G \sim G(n,cn^{-1/m_2(H)})$ conditioned on being $H$-free. Write $\mathbf{x} = \mathbf{x}(G)$. Whp $\delta_\Box(\tilde G/p,{\mathbf{x}} /p) = o(1)$. Additionally, whp there exists a graphon $w$ that maximizes $\mc B_c^H$ and such that $\delta_\Box({\mathbf{x}} / p, w) = o(1)$.
\end{lemma}

\begin{proof}
    We first prove that whp $\delta_\Box(\tilde G/p,{\mathbf{x}}(G) /p) = o(1)$. This is equivalent to proving that whp
    \begin{equation}\label{eq:cut metric G vs bar x}
    \delta_\Box(\tilde G,{\mathbf{x}}(G)) = o(p),    
    \end{equation}
    which follows almost immediately from \cref{cor:bar x almost maximizes B}. Indeed, by definition,
    \begin{align*}
    \delta_\Box(\tilde G, \mathbf{x}) & \leq \sup_{S,T \subseteq [0,1]}|\tilde G(S \times T) - \mathbf{x}(S \times T)|\\
    & = \frac{1}{n^2} \max_{S,T \subseteq [n]} \left| |\{(i,j) \in S \times T : ij \in E(G)\}| - \sum_{i \in S, j \in T} \mathbf{x}_{ij} \right|.
    \end{align*}
    We observe that the collection $\mc A \subseteq 2^{E(K_n)}$ of sets of the form $E_{S,T} = \{ij : (i,j) \in S\times T\}$ with $S,T \subseteq [n]$ satisfies $\log|\mc A| \leq \log(4^n) = o\left(\binom{n}{2} p\right)$, and so \cref{cor:bar x almost maximizes B} implies that whp
    \[
    \max_{A \in \mc A} \left| |E(G) \cap A| - \sum_{ij \in A} \mathbf{x}_{ij} \right| = o(n^2p),
    \]
    implying \eqref{eq:cut metric G vs bar x}.

    We now prove that whp $\mathbf{x} / p$ is close to a maximizer of $\mc B_c^H$. Let $\varepsilon >0$ and let $\mc M_\varepsilon \subseteq \mc W$ be the set of all graphons at distance at most $\varepsilon$ from a maximizer of $\mc B_c^H$. It suffices to prove that whp $\mathbf{x}/p \in \mc M_\varepsilon$. From the compactness of $\mc W$ and the upper-semicontinuity of $\mc B_c^H$ there exists some $\delta>0$ such that for every $w \in \mc W \setminus \mc M_\varepsilon$ we have $\mc B_c^H(w)<F_H(c)-\delta$. On the other hand, by \cref{cor:bar x almost maximizes B}, whp
    \begin{align*}
    n^{1/m_2(H)-2}\mc B_{\mc H(H,n),p}(\mathbf{x}) & = n^{1/m_2(H)-2} \max_{z \in [0,1]^{E(K_n)}} \mc B_{\mc H(H,n),p}(z) + o(1)\\
    & \stackrel{\cref{lem:H rate function}}{=} F_H(c) + o(1).
    \end{align*}
    By \cref{clm:discrete to continuous Bethe comparison}, $n^{1/m_2(H)-2}\mc B_{\mc H(H,n),p}(\mathbf{x}) = \mc B_c^H(\mathbf{x}/p) + o(1)$. Therefore, $\mc B_c^H(\mathbf{x}/p) = F_H(c)+o(1)$ whp, as claimed.
\end{proof}

Together, \cref{lem:H rate function,lem:H structure} imply \cref{thm:H non-existence}.

\subsection{The structural transition in non-bipartite $H$-free random graphs}

Let $H$ be strictly $2$-balanced. In particular, the discussion in this section encompasses $H=K_r$, which is the primary case we have in mind. \cref{thm:H non-existence} reduces the log asymptotics of $\Pr[H \nsubseteq G(n,cn^{-1/m_2(H)})]$ to determining $F_{H}(c)$. It similarly reduces studying the large-scale structure of typical $H$-free samples from $G(n,cn^{-1/m_2(H)})$ to understanding the set $\mc M_c^H \coloneqq \{ w \in \mc W : \mc B_c^H(w) = F_H(c) \}$ of maximizers of $\mc B_c^H$.

For the remainder of this section, fix a strictly $2$-balanced $H$ with chromatic number $r \geq 3$. The established results (cf.\ \cite{luczak2000triangle,saxton2015hypergraph,balogh2015independent,kozma2023lower,janson1990poisson}) tell us that as $c\to\infty$, typical samples from $G(n,cn^{-1/m_2(H)})$ conditioned on not containing $H$ are approximately $(r-1)$-partite, whereas when $c \to 0$ such samples do not have a similar large scale structure. We show here how this phenomenon is reflected in $\mc M_c^H$ and thus, by \cref{thm:H non-existence}, in the typical structure of $H$-free samples from $G(n,cn^{-1/m_2(H)})$.

Let $k=e(H)$ and $R=k-1$.

\subsubsection{Small $c$}

In this section we show that for $c$ sufficiently small in terms of $H$, $\mc B_c^{H}$ is uniquely maximized by an explicit constant-density graphon. In particular, typical $H$-free samples from $G(n,p)$ in this regime do not have any notable partite structure.

We write $W$ for Lambert's $W$-function, i.e., the inverse of $f(x)=xe^x$.

The following follows directly from the authors' results in \cite{jenssen2026non}, where it was proved using contraction of a Belief Propagation operator (and with the explicit quantity $c_H = \left(\frac{e|\AutH|}{2k(k-1)}\right)^{1/(k-1)}$). We give here a proof using \cref{thm:H non-existence} and a variational analysis of the functional $\mc B_c^{H}$.

\begin{theorem}\label{thm:small c H-free analysis}
    Let $H$ be strictly $2$-balanced with chromatic number $r \geq 3$. There exists some $c_H > 0$ such that the following holds. For every $c<c_H$, let $w_c$ be the constant graphon
    \[
    w_c \equiv \frac{1}{c} \left( \frac{|\AutH|}{2kR} W \left( \frac{2k R c^R}{|\AutH|} \right) \right)^{1/R}.
    \]
    Then $w_c$ is the unique maximizer of $\mc B_c^{H}$, and
    \[
    \log(\Pr[H \nsubseteq G(n,cn^{-1/m_2(H)})]) = n^{2-1/m_2(H)} \mc B_c^{H}(w_c) + o(n^{2-1/m_2(H)}).
    \]
\end{theorem}

To prove \cref{thm:small c H-free analysis} it suffices to prove that for sufficiently small $c$, $\mc B_c^{H}$ is strictly concave and that $w_c$ is a stationary point of $\mc B_c^{H}$. These facts are established by the next two claims.

\begin{claim}
    For every $c < (|\AutH|/(2kR))^{1/R}$, $\mc B_c^{H}$ is strictly concave.
\end{claim}

\begin{proof}
    We will show that for every two graphons $u\neq v$, the function $g:[0,1]\to\R$ defined by
    \[
    g(t) \coloneqq \mc B_c^{H}(tu+(1-t)v)
    \]
    has negative second derivative on $(0,1)$. Expanding definitions, we see that
    \begin{equation}\label{eq:g second derivative}
    \begin{split}
    g''(t) & = - \frac{c^k}{|\AutH|} \frac{d^2}{dt^2} t(H,tu+(1-t)v) - \frac{c}{2} \iint_{x,y \in [0,1]} \frac{(u(x,y)-v(x,y))^2}{tu(x,y)+(1-t)v(x,y)} dxdy\\
    & \leq - \frac{c^k}{|\AutH|} \frac{d^2}{dt^2} t(H,tu+(1-t)v) - \frac{c}{2} \iint (u-v)^2.
    \end{split}
    \end{equation}
    We now observe that $\frac{d^2}{dt^2} t(H,tu+(1-t)v)$ can be written as a sum of $k(k-1)$ polynomials of the form
    \[
    \int_{[0,1]^{v(H)}} (u(x_a,x_b) - v(x_a,x_b))(u(x_c,x_d) - v(x_c,x_d)) \prod_{\substack{ij \in E(H):\\ij \neq ab,cd}} (tu(x_i,x_j) + (1-t)v(x_i,x_j)) dx_1\cdots dx_{v(H)}.
    \]
    The product on the right is bounded in absolute value by $1$. The Cauchy--Schwarz inequality implies that
    \[
    \int_{[0,1]^{v(H)}} |u(x_a,x_b) - v(x_a,x_b)||u(x_c,x_d) - v(x_c,x_d)| dx_1\cdots dx_{v(H)} \leq \iint (u-v)^2.
    \]
    Together with \eqref{eq:g second derivative} this implies that
    \[
    g''(t) \leq \left( \frac{k(k-1) c^k}{|\AutH|} - \frac{c}{2} \right) \iint (u-v)^2,
    \]
    which is negative as long as $c < (|\AutH|/(2kR))^{1/R}$.
\end{proof}

\begin{claim}
    For every $c>0$, the graphon $w_c$ defined in \cref{thm:small c H-free analysis} is a stationary point of $\mc B_c^{H}$.
\end{claim}

\begin{proof}
    We will prove that for every graphon $u$, the function
    \[
    g(t) \coloneqq \mc B_c^{H}(w_c+tu)
    \]
    (which is always defined in a neighborhood of $0$) has a stationary point at $0$. Set
    \[
    q \coloneqq \frac{1}{c} \left( \frac{|\AutH|}{2kR} W \left( \frac{2k R c^R}{|\AutH|} \right) \right)^{1/R}
    \]
    as the value that $w_c$ takes. The definitions imply that
    \begin{align*}
        g'(0) & = - \frac{c^{k}}{|\AutH|} k q^R \iint u(x,y) dxdy - \frac{c}{2} \log q \iint u(x,y) dxdy\\
        & = - \frac{c}{2} \left( \frac{2kc^R}{|\AutH|} \times \frac{|\AutH|}{2kRc^R} W \left( \frac{2k R c^R}{|\AutH|} \right) + \log q \right) \iint u(x,y) dxdy\\
        & = - \frac{c}{2} \left(  \frac{1}{R} W \left( \frac{2k R c^R}{|\AutH|} \right) + \log q \right) \iint u(x,y) dxdy.
    \end{align*}
    The final quantity is equal to zero since
    \begin{equation}\label{eq:R W q equality}
    \frac{1}{R} W \left( \frac{2k R c^R}{|\AutH|} \right) + \log q = 0.
    \end{equation}
    Indeed, writing $z \coloneqq \frac{2k R c^R}{|\AutH|}$, we have $q^R = W(z)/z$. Since, by definition, $W(z)e^{W(z)}=z$, it follows that $W(z)/z = e^{-W(z)}$. Consequently, $\log q = -W(z)/R$, implying \eqref{eq:R W q equality}.
\end{proof}

\subsubsection{Large $c$}\label{sec:large c analysis}

As mentioned above, as $c \to \infty$, typical $H$-free samples from $G(n,cn^{-1/m_2(H)})$ are approximately $(r-1)$-partite. In this section we show that the corresponding structure is reflected in graphons maximizing $\mc B_c^{H}$, and hence in typical $H$-free samples from $G(n,cn^{-1/m_2(H)})$.

We define $w_r^{1,0}$ as the balanced $(r-1)$-partite graphon (with internal density $0$ in each part and crossing density $1$ between parts). We prove the following on the structure of graphons maximizing $\mc B_c^{H}$, with $c$ large.

\begin{theorem}\label{thm:large c graphon H-free analysis}
    Let $H$ be strictly $2$-balanced with chromatic number $r \geq 3$. Suppose that $\{m_c\}_{c>0}$ is a family of graphons such that for each $c$, $\mc B_c^{H}(m_c) = F_{H}(c)$. Then $\lim_{c\to\infty}\delta_\Box(m_c,w_r^{1,0}) = 0$.
\end{theorem}

The proof of \cref{thm:large c graphon H-free analysis} is as follows. We will first argue that as $c$ increases, the density of $H$ in $m_c$ approaches zero. We will also show an entropy inequality for $m_c$. These will later be combined with the stability version of the Erd\H{o}s--Stone--Simonovits theorem and a compactness argument to prove the theorem.

\begin{claim}
    For every $c$,
    \[
    t(H,m_c) \leq \frac{|\AutH|}{2(r-1)c^R}
    \]
    and
    \[
    I(m_c) \leq \frac{1}{2(r-1)}.
    \]
\end{claim}

\begin{proof}
    Since $m_c$ maximizes $\mc B_c^{H}$, there holds
    \[
    - \frac{c^k}{|\AutH|}t(H,m_c) \geq - \frac{c^k}{|\AutH|}t(H,m_c) - c I(m_c) = \mc B_c^{H}(m_c) \geq \mc B_c^{H}(w_r^{1,0}) = - \frac{c}{2(r-1)},
    \]
    where the final equality follows from plugging $w_r^{1,0}$ into the definition of $\mc B_c^{H}$ while noting that $t(H,w_r^{1,0})=0$. The first assertion in the claim follows.

    Arguing similarly,
    \[
    - c I(m_c) \geq \mc B_c^{H}(w_r^{1,0}) = - \frac{c}{2(r-1)},
    \]
    and the claim follows.
\end{proof}

We now prove \cref{thm:large c graphon H-free analysis}.

\begin{proof}
    Since the graphon space is compact with respect to the cut metric, it suffices to prove that every convergent subsequence of $\{m_c\}$ converges to $w_r^{1,0}$.
    
    Let $\{c_n\}_{n\in\N}$ be a sequence with $c_n\to\infty$ and such that $m_{c_n}$ converges in cut metric to $w \in \mc W$. Since $t(H,\cdot)$ is continuous in cut metric and $I(\cdot)$ is lower semicontinuous, the previous claim implies that
    \[
    t(H,w) = \lim_{n\to\infty} t(H,m_{c_n}) \leq \lim_{n\to\infty} \frac{|\AutH|}{2(r-1)c_n^R} = 0
    \]
    and that
    \[
    I(w) \leq \liminf_{n\to\infty} I(m_{c_n}) \leq \frac{1}{2(r-1)}.
    \]
    We now claim that $w = \mathbf 1_{w>0}$ (i.e., $w$ takes the value $1$ almost everywhere on its support). Indeed, if we let $u = \mathbf 1_{w>0}$, then $t(H,u) = 0$. Therefore, the graphon version of the Erd\H{o}s--Stone--Simonovits theorem implies that $\iint u(x,y)dxdy \leq 1- 1/(r-1)$. Additionally,
    \[
    2 I(w) = \iint_{w=0}1 + \iint_{w>0}w(x,y)\log(w(x,y)) + 1 - w(x,y) dxdy \geq 1 - \iint u(x,y) dxdy \geq \frac{1}{r-1}.
    \]
    It follows that $\iint u(x,y) dxdy = 1-1/(r-1)$, and additionally, that $w=1$ almost everywhere on its support. Thus, $u = w$ and $\iint w(x,y)dxdy = 1 - 1/(r-1)$.

    By the stability version of the Erd\H{o}s--Stone--Simonovits theorem for graphons (cf.\ \cite[Theorem 16.23]{lovasz2012large}) we have $\delta_\Box(w,w_r^{1,0})=0$, completing the proof.
\end{proof}

\subsubsection{Conjectured structure of optimizers}

The small-$c$ and large-$c$ examples lead to a natural family of candidates for optimizing graphons $w_c$: Balanced $(r-1)$-step graphons. Specifically, for $\alpha,\beta \in [0,1]$, let $w_r^{\alpha,\beta}$ be the graphon with constant value $\beta$ on ``diagonal'' squares of the form $((i-1)/(r-1),i/(r-1)) \times ((i-1)/(r-1),i/(r-1))$ and constant value $\alpha$ everywhere else. We conjecture that optimizers of $\mc B_c^{K_r}$ can always be found in this family, as well as the stronger conjecture that \textit{all} optimizers are in this family.

\begin{conjecture}\label{con:block optimizers}
    For every $r \geq 3$ and $c>0$,
    \[
    F_{K_r}(c) = \max_{\alpha,\beta \in [0,1]} \mc B_c^{K_r} (w_r^{\alpha,\beta}).
    \]
\end{conjecture}

\begin{conjecture}\label{con:block structure}
    For every $r \geq 3$ and $c>0$,
    \[
    \{ w \in \mc W : \mc B_c^{K_r}(w) = F_{K_r}(c) \} \subseteq \{ w_r^{\alpha,\beta} : \alpha,\beta \in [0,1] \},
    \]
    where we view graphons as equivalent under measure-preserving relabelings.
\end{conjecture}

The results in \cite{jenssen2026non} mentioned above imply that \cref{con:block structure} holds when $c < \left(\frac{e (r-2)!}{R}\right)^{1/R}$. As mentioned in the introduction, Kwan and Pham \cite{kwan2026nonlinear} proved the $r=3$ case of \cref{con:block structure}. It is possible that a more detailed structural analysis along the lines of \cref{sec:large c analysis} might be sufficient to prove \cref{con:block optimizers} for large $c$. However, proving the conjecture for all $r$ and $c$ might require additional ideas.

\section*{On the use of AI in this paper}

\cref{lem:Bethe FE lower bound} was originally proved using a discrete version of the random greedy algorithm in \cref{subsec:BP-guided}. The proof was complete and correct as far as we could tell, but it was rather technical. The authors subsequently discovered the conceptually simpler proof using the BP guided random greedy process. However, when writing the details we found a mistake in our argument for comparing $a^{\textrm{ind}}_v(t)$ and $a^{\textrm{sh}}_v(t)$. ChatGPT 5.6 Sol was used to find a fix which led to the introduction of intermediate clock processes and the proof of \cref{lem:fourcorners}. The statement of \cref{lem:influence-same-label} had to be updated accordingly, but the proof remained essentially unchanged. With the help of GPT 6 Astra we tried to trace the idea behind \cref{lem:fourcorners} in the literature and we found the clearest precedent in the work of Chatterjee~\cite{Chatterjee2008}. There the author uses discrete mixed differences, analogous to the four-term expression in~\eqref{eq:Gamma-swap}, to describe interactions between coordinates in a coordinate-resampling approach to normal approximation.

All other proofs in the paper are the authors' own, and the authors also wrote the manuscript. Microsoft Copilot and ChatGPT 6 were used as editing, literature search, and proofreading tools, including checking mathematical content. Needless to say, the authors take full responsibility for the paper.

\section*{Acknowledgments}
MJ is supported by a UK Research and Innovation Future Leaders Fellowship MR/W007320/2. WP is supported in part by NSF grant DMS-2348743. MS was supported by NSF grant DMS-2349024 and by ISF Beresheet grant 4287/26.
\bibliographystyle{amsplain.bst}
\bibliography{triangle}

\end{document}